\documentclass[journal]{IEEEtran}
\usepackage{amsmath,amsfonts}
\usepackage{algorithmic}
\usepackage{algorithm}
\usepackage{array}
\usepackage[caption=false,font=normalsize,labelfont=sf,textfont=sf]{subfig}
\usepackage{textcomp}
\usepackage{stfloats}
\usepackage{url}
\usepackage{color}
\usepackage{verbatim}
\usepackage{graphicx}
\usepackage{cite}

\newtheorem{theorem}{Theorem}[section]
\newtheorem{lemma}[theorem]{Lemma}

\newtheorem{corollary}[theorem]{Corollary}
\newtheorem{example}[theorem]{Example}
\newtheorem{remark}[theorem]{Remark}
\newenvironment{proof}{\noindent{\em Proof:}}{\quad \hfill$\Box$\vspace{2ex}}
\floatstyle{ruled}
\floatname{algorithm}{Algorithm}
\usepackage{booktabs}
\def \bR {\mathbb{R}}
\def \bN {\mathbb{N}}

\def \bx {{\boldsymbol x}}

\def \bb {{\boldsymbol b}}

\def \sign  {\,{\rm sign}\,}
\def \prox  {\,{\rm Prox}\,}

\begin{document}

\title{The Proximal Operator of the Piece-wise Exponential Function and Its Application in Compressed Sensing}

\author{Yulan Liu, Yuyang Zhou, and Rongrong Lin}

\markboth{ }%
{Shell \MakeLowercase{\textit{et al.}}: The Proximal Operator of the Piece-wise Exponential Function and Its Application in Compressed Sensing}


\maketitle

\begin{abstract}
This paper characterizes the proximal operator of the piece-wise exponential function $1\!-\!e^{-|x|/\sigma}$ with a given shape parameter $\sigma\!>\!0$, which is a popular nonconvex surrogate of $\ell_0$-norm in support vector machines, zero-one programming problems, and compressed sensing, etc. Although Malek-Mohammadi et al. [IEEE Transactions on Signal Processing, 64(21):5657--5671, 2016] once worked on this problem, the expressions they derived were regrettably  inaccurate. In a sense, it was lacking a case. Using the Lambert W function and an extensive study of the piece-wise exponential function, we have rectified the formulation of the proximal operator of the piece-wise exponential function in light of their work. We have also undertaken a thorough analysis of this operator. Finally, as an application in compressed sensing, an iterative shrinkage and thresholding algorithm (ISTA) for the piece-wise exponential function regularization problem is developed and fully investigated.
A comparative study of ISTA with nine popular non-convex penalties in compressed sensing demonstrates the advantage of the piece-wise exponential penalty.
\end{abstract}


\section{Introduction}\label{Introduction}

The notion of the proximal operator of a convex function, first introduced by Moreau in his seminal work \cite{Moreau65}, is kicking and alive. This fundamental regularization process gave birth five years later to the  so-called proximal  algorithm
by Martinet \cite{Martinet70}, followed by its extension in Rockafellar \cite{Rockafellar76}  for studying variational inequalities associated with maximal monotone inclusions. In the last three decades, it has been successfully applied to a wide variety of situations: convex optimization \cite{Donoho1992, Rockafellar76} and nonmonotone operators \cite{Attouch13,Bolte17,Combettes04}. Recently, solving large and huge-scale problems arising in a wide spectrum of disparate fundamental applications provide extra motivation to develop further the study of the proximal algorithm in a nonconvex and possibly nonsmooth setting. However, it is well known that evaluating the proximal operator of a function is basic and key important for proximal algorithms, which is expected to solve a specific optimization problem that is typically easier than the original problem \cite{Parikh2014}. 
This sparked interest in the characterization of some important nonconvex functions' proximal operators, which are frequently observed in sparse optimization  \cite{Fan2001,Gong2013,Xu2023,Zhang2010}.
Many widely known methods including the proximal gradient algorithm, alternating direction method of multipliers, iterative shrinkage and thresholding algorithm (ISTA), majorization-minimization and iteratively reweighed least squares fall into the proximal framework \cite{Beck2017,Polson2015}.

Sparse optimization problems arise in many fields of science and engineering, such as compressed sensing, image processing, statistics, and machine learning, etc. The so-called $\ell_0$-norm, which counted the nonzero components of a vector, is a natural penalty function to promote sparsity.  Numerous studies on $\ell_0$-norm penalty optimization problem have been widely investigated in the literature \cite{Blumensath2008,Fan2020,Shen2016,Wright2022}.
However, such a nonconvex problem is NP-hard \cite{Blumensath2008}. 
There are a number of $\ell_0$-norm  approximations listed in the literature \cite[Table 1]{Xu2023}.
The $\ell_1$-norm regularizer has received a great deal of attention because it is
  a continuous and convex surrogate of $\ell_0$-norm \cite[Theorem 2.11]{Wright2022}. Although it comes close to the $\ell_0$-norm, the $\ell_ 1$-norm frequently leads to problems with excessive punishment. To remedy this issue,
non-convex sparsity-inducing penalties have been employed to better approximate the $\ell_0$-norm and enhance sparsity, and hence have received considerable attention in sparse learning.
  Recent theoretical studies have shown their superiority to the convex counterparts in a variety of sparse learning settings, including the bridge $\ell_p$-norm penalty \cite{Foucart2009}, capped $\ell_1$ penalty (CaP) \cite{Bian2020,Jiang2015,Peleg2008,Zhang2010b}, transformed $\ell_1$ penalty (TL1) \cite{Zhang2017,Zhang2018}, log-sum penalty (Log) \cite{Candes2008}, minimax concave penalty (MCP) \cite{Zhang2010}, smoothly clipped absolute derivation (SCAD) \cite{Fan2001}, the difference of $\ell_1$ and $\ell_2$ norms  \cite{Lou2018,Yin2015}, the ratio of $\ell_1$ and $\ell_2$ norms  \cite{Tao2022,Yin2014} and piece-wise exponential (PiE) function in \cite{Bradley1998,Le2015,Mangasarian1996}.

In this paper, we focus on the PiE penalty function
\begin{equation}\label{PiE}
f_{\sigma}(x)=1-e^{-\frac{|x|}{\sigma}},\ x\in\bR
\end{equation}
where   $\sigma>0$ is  a given shape parameter, which is
also a continuous approximation of the $\ell_0$-norm in 
\cite[Equation (19)]{Mangasarian1996} and  \cite[Equation (7)]{Bradley1998}.  
The PiE is also called an exponential-type penalty (ETP or EXP) in some references \cite{Le2015,Xu2023}.
The parameter $\sigma$ determines the quality of the approximation:  the smaller value of $\sigma$, the closer behavior of to $\ell_0$-norm  (see, Fig. \ref{FigurePiE}). 
The PiE penalty was successfully applied in the support vector machines \cite{Fung2002}, zero-one programming problems \cite{Lucidi2010, Rinaldi2009}, image reconstruction \cite{Trzasko2008}, compressed sensing \cite{Chen2014,Le2015,Nguyen2015}, and the low-rank matrix completion \cite{Yan2022}, etc. 
Numerous numerical techniques are used to resolve the PiE penalized optimization.
For example,
Bradley \cite{Bradley1998,Bradley1998b} demonstrated how to minimize the PiE penalty for feature selection by solving a series of linear algorithms. 
However, this method is somewhat constrained for large problems. 
Difference of convex functions (DC) Algorithm was developed in \cite{Le2015,Nguyen2015} for the PiE penalty.
In their work, the DC decomposition of PiE function is $f_{\sigma}(x)=g(x)-h(x)$ with convex functions $g(x)=\frac{1}{\sigma} |x|$ and $h(x)=\frac{|x|}{\sigma}-1+e^{-\frac{|x|}{\sigma}}$.
According to numerical experiments in \cite[Table 7]{Le2015} for real datasets and in \cite[Figures 1-3]{Nguyen2015} for compressed sensing, the DC algorithm involving PiE penalty achieves very impressive numerical performance among the ones with PiE, $\ell_p$-norm, SCAD, CaP, and Log, etc.
PiE is always among the top performers in RIP and non-RIP categories alike.

To develop the proximal algorithm for the PiE penalty-based optimization problem, it is necessary to inscribe the proximal operator of PiE function.
As far as we know, Malek-Mohammadi et al. studied the issue \cite{Malek2016}. Unfortunately, they derive the incorrect  proximal operator of PiE \cite[Equation (25)]{Malek2016}. Example \ref{counterexample} provides a counterexample. Therefore, the main purpose of this paper is to characterize the proximal operator of PiE function and to provide its thorough analysis.

\begin{figure}[!tbp]
\begin{center}
\includegraphics[scale=0.6]{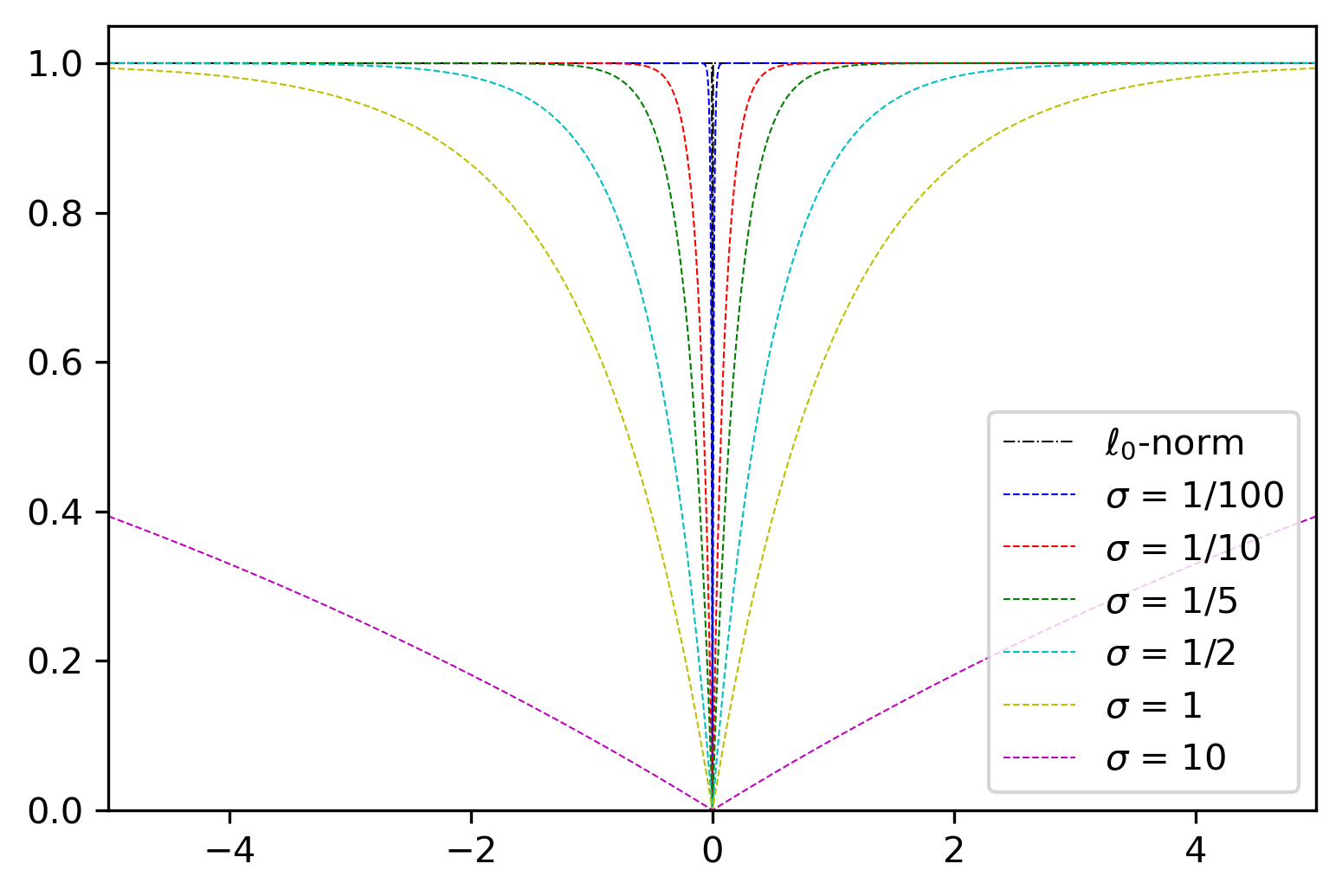}
\caption{The PiE with $\sigma=\{1/100,1/10,1/5,1/2,1,10\}$ and the $\ell_0$-norm.}
\label{FigurePiE}
\end{center}
\end{figure}

The remainder of the paper is organized as follows.
In Section \ref{SecPiEProx}, we first propose some properties of the proximal operator of PiE function. Based on these properties, its implicit expression is given. The proximal operator of the PiE function can be either continuous or discontinuous depending upon shape parameters $\sigma$ and the regularization
parameter (see, Theorems \ref{ProximalCase1} and \ref{Coratau}).  As an application in compressed sensing, numerical experiments are conducted and analyzed in Section \ref{Section:CS}. Subsection \ref{Subsection:ISTA} aims at developing ISTA with the PiE penalty function, as will be given in Algorithm \ref{AlgoISTA}.
The influence of three parameters (the regularization
parameter, the shape parameter, and the step size) in ISTA  with PiE penalty will be discussed in Subsection \ref{Subsection:Experiments}.
In Subsection \ref{Subsection:compare}, we will compare the reconstruction success rate and time
efficiency of ISTA  for  some  popular separable $\ell_0$-norm surrogate penalties. We conclude the paper in Section \ref{Section:Conclusion}.

  \section{Proximal operator of PiE}\label{SecPiEProx}
  The proximal operator with respect to a  function $f:\bR^n\!\to \!\bR$ is defined by for any $\bx_0\in \bR^n$
\begin{align}\label{ProxDef0}
\prox_{\mu\lambda f}(\bx_0):=\arg\min_{\bx\in\bR^n} \lambda f(\bx)\!+\!\frac{1}{2\mu}\|\bx-\bx_0\|^2, 
\end{align}
where $\mu>0$ and $\lambda>0$ are given parameters.
As mentioned above in Section \ref{Introduction}, characterization of $\prox_{\mu\lambda f}$ is crucial to the proximal algorithm. When $f$ is a  separate function,
by \cite[Theorem 6.6]{Beck2017}, it is sufficient to  deduce the characterization of $\prox_{\mu\lambda f}$ in \eqref{ProxDef0} when $n=1$, i.e.,  for any $x_0\in \bR$
\begin{align*}
\prox_{\mu\lambda f}(x_0):=\arg\min_{x\in\bR} \lambda f(x)+\frac{1}{2\mu} (x-x_0)^2.
\end{align*}
The proximal operator of some commonly used nonconvex separable $\ell_0$-norm surrogates, such as SCAD \cite{Fan2001}, MCP \cite{Zhang2010}, CaP \cite{Gong2013}, Log \cite{Prater2022}, TL1 \cite{Zhang2017} are well-known.

In this work, we study the proximal operator of PiE penalty function, defined as in \eqref{PiE}, i.e., for any $x_0\in \bR$
\begin{align}\label{Prox_fDef}
\prox^{(\sigma)}_{\mu\lambda  }(x_0)
:={\rm arg}\min\limits_{x\in \bR}\lambda (1-e^{-\frac{|x|}{\sigma}})+\frac{1}{2\mu}(x-x_0)^2.
\end{align}
Denote
\[
 L(x,x_0):=\lambda (1-e^{-\frac{|x|}{\sigma}})+\frac{1}{2\mu}(x-x_0)^2, \mbox{ for any } x\in \bR.
\]
Notice that $L(x,x_0)\!\to\!+\infty$ as $|x|\to \!+\!\infty$, the optimization problem in \eqref{Prox_fDef} has a nonempty optimal solution by the Weierstrass Theorem \cite[Theorem 2.14]{Beck2017}. Hence, 
the proximal operator $\prox^{(\sigma)}_{\mu\lambda  }$ is well defined.
Since  $L(x,-x_0)\!=\!L(-x,x_0)$ for every $x\!\in\!\bR$,
  the following   anti-symmetry  of  $\prox^{(\sigma)}_{\mu\lambda f}$ is easy to be checked along with \eqref{Prox_fDef}.

  \begin{lemma}\label{Antisym} 
  Let $\mu>0,\lambda>0$ and $\sigma>0$. For any $x_0\!\in\!\bR$,  $\prox^{(\sigma)}_{\mu\lambda  }(x_0)\!=\!-\prox^{(\sigma)}_{\mu\lambda }(-x_0)$. 
  \end{lemma}
In what follows, let $\mu>0,\lambda>0$ and $\sigma>0$ be given.
By Lemma \ref{Antisym}, we know that
\begin{equation}\label{ProxTmulambdasigma}
\prox^{(\sigma)}_{\mu\lambda }(x_0)={\rm sign}(x_0)\mathcal{T}_{\mu\lambda}^{(\sigma)}(x_0),\  \mbox{for any }x_0\in \bR,
\end{equation}
 where
   \begin{align}\label{MainProb}
   \mathcal{T}_{\mu\lambda}^{(\sigma)}(x_0):={\rm arg}\min\limits_{x\geq 0}\Big\{ \lambda (1-e^{-\frac{x}{\sigma}})+\frac{1}{2\mu}(x-|x_0|)^2 \Big\}.
   \end{align} 
   So,  it is crucial to drive the expression of
   $\mathcal{T}_{\mu\lambda}^{(\sigma)}(x_0)$ in order to characterize $\prox^{(\sigma)}_{\mu\lambda }(x_0)$.
    For convenience,  for any $x_0\in\bR$, we denote 
   \begin{align}\label{LDef}
   \widehat{L}(x,x_0):=\lambda (1-e^{-\frac{x}{\sigma}})+\frac{1}{2\mu}(x-|x_0|)^2.
   \end{align}
    Before moving on, let us recall the Lambert W function which is a multi-valued inverse function with plenty of applications in areas like the theory of projectile motion, wave physics, cell growth models, solar wind, enzyme kinetics, plasma and quantum physics, general relativity, prey and parasite models, Twitter events, and so on \cite{Mezo2022}. The Lambert W function $W(x)$ is a set of solutions of the equation 
   $$
   x = W (x)e^{W(x)},\mbox{ for any } x\in [-\frac{1}{e},+\infty).
   $$
  The function $ W(x)$ is  single-valued for  $x\geq 0$ or $x=-\frac{1}{e}$, and is double-valued for
 $-\frac{1}{e}< x< 0 $ (see Fig. \ref{LambertW}). To discriminate between the two branches for $-\frac{1}{e}\leq x< 0 $, we use the same notation as in  \cite[Section 1.5]{Mezo2022}
 and denote the branch satisfying $W(x)\geq -1$ and $W(x)\leq -1$
 by $W_0(x)$ and $W_{-1}(x)$, respectively.
   Such a function is a built-in function in Python 
   (\url{https://docs.scipy.org/doc/scipy/reference/generated/scipy.special.lambertw.html}).
   We refer the readers to the recent monograph \cite[Section 1.5]{Mezo2022} on the Lambert W function to learn more details. 
  
  \begin{figure}[!tbp]
\begin{center}
\includegraphics[scale=0.6]{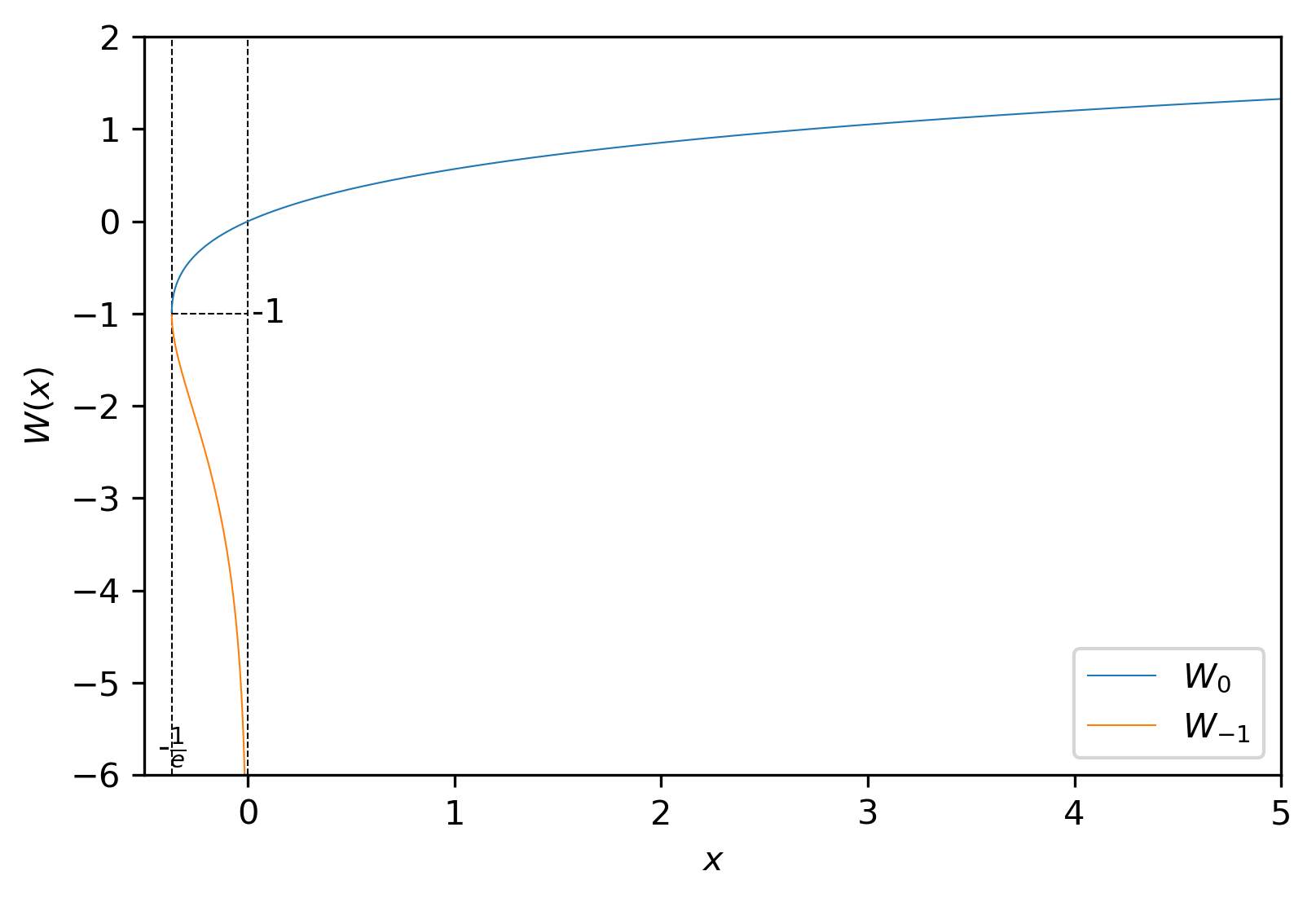}
\caption{Two main branches of the Lambert W function.}
\label{LambertW}
\end{center}
\end{figure}
   
   Malek-Mohammadi et al. \cite{Malek2016} proposed a new algorithm for the recovery of sparse vectors from underdetermined measurements based on  a successively accuracy increasing approximation of the $\ell_{0}$-norm for compressed sensing. As a by-product, the expression of $ \mathcal{T}_{\mu\lambda}^{(\sigma)}$ is derived in \cite[(25)]{Malek2016} with
   \begin{align}\label{ErrorFormula}
   \mathcal{T}_{\mu\lambda}^{(\sigma)}(x_0)=\left\{\begin{array}{cl}
    \{0\}, & \quad| x_{0}  | < \sigma (1+\ln( \frac{\mu \lambda}{\sigma ^{2}})),   \\ 
    \{0\}, & \quad\widehat{L}(x_{1},x_{0})>\widehat{L}(0,x_{0}), \\
     \{0,x_1\}, & \quad\widehat{L}(x_{1},x_{0})=\widehat{L}(0,x_{0}), \\
    \{x_1\}, &\quad {\rm otherwise},
    \end{array}\right.
   \end{align}
   where $x_1=\sigma W_{0}(- \frac{\mu \lambda}{\sigma ^{2}}e^{-\frac{|x_{0}|}{\sigma }})+| x_{0}|$ and $\widehat{L}(\cdot,x_0)$ defined as in \eqref{LDef}. 
   Unfortunately, the following example illuminate  this closed form expression of $\mathcal{T}_{\mu\lambda}^{(\sigma)}$ in \eqref{ErrorFormula} is inaccurate. 
  
   \begin{example}\label{counterexample}{\bf (Counter-example)}
   Set $\mu=\lambda=1$ and $\sigma = 2$. Taking $x_0 = \frac{1}{4}$.
   After simple calculation, we know  $x_1 \!=\! - 0.3438$, $ - 0.0113=\widehat{L}(x_1, x_0)< \widehat{L}(0, x_0) = \frac{1}{32}$ and $- 0.7726=\sigma (1+\ln( \frac{\mu \lambda}{\sigma ^{2}}))<|x_0|$. Therefore, 
   $\mathcal{T}_{\mu\lambda}^{(\sigma)}(x_0)=\{x_1\}=\{- 0.3438\}$ from  \eqref{ErrorFormula}, which 
   contracts to the constraint $x\geq 0$ in the optimization problem \eqref{MainProb}.
   \end{example}

\begin{remark}
In \cite{Malek2016}, the expression of $\mathcal{T}_{\mu\lambda}^{(\sigma)}$ in \eqref{ErrorFormula} is obtained only on the basis of the optimal necessity condition. However, the optimal optimization is a nonconvex problem. Hence, the related analysis in \cite[Page 13]{Malek2016} is not sufficient. 
  It is worth pointing out that equation \eqref{ErrorFormula} is correct when $\frac{\mu\lambda}{\sigma^2}>1$,  which is equivalent to  Theorem \ref{ProximalCase2} later. However,
   in \eqref{ErrorFormula}, one must compare the values $\widehat{L}(x_{1},x_{0})$ and $\widehat{L}(0,x_{0})$ every time to determine  $\mathcal{T}_{\mu\lambda}^{(\sigma)}(x_0)$ whenever $| x_{0}  |\ge \sigma (1+\ln( \frac{\mu \lambda}{\sigma ^{2}}))$, while in Theorem \ref{ProximalCase2} only when $\sigma(1+\ln\frac{\mu\lambda}{\sigma^2})\leq |x_0|\leq \frac{\mu\lambda}{\sigma}$ the values $\widehat{L}(x_{1},x_{0})$ and $\widehat{L}(0,x_{0})$ need to be compared, or only to find a threshold value $\bar{\tau}_{\mu\lambda,\sigma}$ according to Theorem \ref{Coratau}.
  When $\frac{\mu\lambda}{\sigma^2}\!\leq\! 1$, $\mathcal{T}_{\mu\lambda}^{(\sigma)}$ should be as in \eqref{Case1}, see Theorem \ref{ProximalCase1}.
\end{remark}

  \subsection{Expressions of the proximal operator for PiE}

  In this subsection, we will focus on building the correct expression 
  of $\mathcal{T}_{\mu\lambda}^{(\sigma)}$. To this end, we need some lemmas, which are needed later.

   \begin{lemma}
   	 For any $x_0\in \mathbb{R}$, it holds that $\mathcal{T}_{\mu\lambda}^{(\sigma)}(x_0)\neq \emptyset$.
   	\end{lemma}
   \begin{proof}
   Note that $\lim\limits_{x\!\to \!+\infty}\widehat{L}(x,x_0)\!=\!+\infty$. 
   $\mathcal{T}_{\mu\lambda}^{(\sigma)}(x_0)\!\neq\! \emptyset$ follows from
    the Weierstrass Theorem \cite[Theorem 2.14]{Beck2017}.
   \end{proof}
  
  \begin{lemma}\label{ThresholdLemma}
   For any $\mu>0, \lambda>0$ and $\sigma>0$, it holds 
    that $1+\ln(\frac{\mu\lambda}{\sigma^2})\leq\frac{\mu\lambda}{\sigma^2}$. 
   \end{lemma}
   \begin{proof} Let $g(t)=1+\ln t-t$ for any $t\in \mathbb{R}$. Obviously, $g(1)=0$
   	and $g'(t)=\frac{1}{t}-1$. Therefore, $g(t)$ is increasing on the interval $(0,1)$ and decreasing on the interval $(1,\infty)$. Hence, $g(t)\!\leq\! g(1)\!=\!0$ for any $t\!>\!0$, which by taking $t=\frac{\mu\lambda}{\sigma^2}$ implies  $1+\ln(\frac{\mu\lambda}{\sigma^2})\leq\frac{\mu\lambda}{\sigma^2}$ .
   	\end{proof}

   \begin{lemma}\label{LemmaNecCondi}
   	Given $x_0\in \mathbb{R}$.
   	If $|x_0|> \frac{\mu\lambda}{\sigma}$, one has 
    \[
    \mathcal{T}_{\mu\lambda}^{(\sigma)}(x_0)=\{\sigma W_0(-\frac{\mu\lambda}{\sigma^2}e^{-\frac{|x_0|}{\sigma}})+|x_0|\}.
    \]
   \end{lemma}
   \begin{proof}
   Take any point $x^*$ from  $\mathcal{T}_{\mu\lambda}^{(\sigma)}(x_0)$. 
   	Since $|x_0|> \frac{\mu\lambda}{\sigma}$, we have that $\frac{\lambda}{\sigma}-\frac{|x_0|}{\mu}<0$.
   Applying the optimal necessary condition \cite[Theorem 10.1]{Rockafellar2009} to the optimizaiton problem \eqref{MainProb}, we know that  $x^*$ satisfy the following inclusion:
  \begin{align}\label{NecCondi}
  0\in \frac{1}{\mu}(x^*-|x_0|)+\frac{\lambda}{\sigma}e^{-\frac{x^*}{\sigma}}+\mathcal{N}_{\mathbb{R}_+}(x^*).
 \end{align}
 where ${\mathcal{N}}_{\mathbb{R}_+}(x^*)$ is a norm cone to $\bR_{+}$ at $x^*$, i.e.,
 \begin{align}\label{NormCone}
     {\mathcal{N}}_{\mathbb{R}_+}(x^*)=\{v\in \bR\,|\, v(x-x^*)\leq 0, \mbox{ for any } x\in \bR_{+}\}.
 \end{align}
 Obviously, $\mathcal{N}_{\mathbb{R}_+}(0)=\mathbb{R}_-$.
 Hence, $x^*\neq 0$ from \eqref{NecCondi} and the fact $\frac{\lambda}{\sigma}-\frac{|x_0|}{\mu}<0$. So, $x^*>0$ and then
  $\mathcal{N}_{\mathbb{R}_+}(x^*)=\{0\}$ by \eqref{NormCone}.  Again from \eqref{NecCondi},
 it follows that
 \[
  \frac{1}{\mu}(x^*-|x_0|)+\frac{\lambda}{\sigma}e^{-\frac{x^*}{\sigma}}=0,
  \]
  namely,  
  $$
  \frac{x^*-|x_0|}{\sigma}e^{\frac{x^*-|x_0|}{\sigma}}=-\frac{\mu\lambda}{\sigma^2}e^{-\frac{|x_0|}{\sigma}}.
  $$
  By Lemma \ref{ThresholdLemma} and  $|x_0|> \frac{\mu\lambda}{\sigma}$, $|x_0|\geq \sigma (1+\ln(\frac{\mu\lambda}{\sigma^2}))$ and then
  $-\frac{\mu\lambda}{\sigma^2}e^{-\frac{|x_0|}{\sigma}}\in [-\frac{1}{e},0)$.
  Hence,  together with the property of Lambert W function in \cite{Corless1996}, 
  $x^*=x_1$ or $x^*=x_2$ where
   \begin{align}\label{WSolution}
   \left\{
   \begin{array}{l}
	x_1:=\sigma W_0(-\frac{\mu\lambda}{\sigma^2}e^{-\frac{|x_0|}{\sigma}})+|x_0|,\\
	x_2:=\sigma W_{-1}(-\frac{\mu\lambda}{\sigma^2}e^{-\frac{|x_0|}{\sigma}})+|x_0|.
   \end{array}
   \right.
   \end{align}
   From arguments made in \cite[Page 13]{Malek2016}, we know $\widehat{L}(x_1,x_0)\!\leq\! \widehat{L}(x_2,x_0)$. Thus, $x^*\!=\!x_1$.
   	\end{proof}

   \begin{lemma}\label{shrinkage} Let  $ x_0\neq 0$. Then  $\mathcal{T}_{\mu\lambda}^{(\sigma)}(x_0)\subseteq[0,|x_0|)$.

   \end{lemma}
\begin{proof}
Since
$
\widehat{L}'(x,x_0)\!=\!\frac{\lambda}{\sigma} e^{-\frac{x}{\sigma}}+\frac{x-|x_0|}{\mu},
$
for any $x\!>\!0$, it holds that $\widehat{L}'(x,x_0)>0$ when $x>|x_0|$.
Notice that $\widehat{L}'(|x_0|,x_0)=\frac{\lambda}{\sigma} e^{-\frac{|x_0|}{\sigma}}>0$. By the sign-preserving property of continuous functions $\widehat{L}'(x,x_0)$ at the point $|x_0|$, there exists a sufficient small $\delta>\!0\!$ such that $\widehat{L}'(x,x_0)\!>\!\frac{\widehat{L}'(|x_0|,x_0)}{2}\!>\!0$ for any $x\!\in\!(|x_0|-\delta,|x_0|+\delta)$.
Consequently,
$\widehat{L}$ is strictly increasing on $[|x_0|-\delta,+\infty)$ and then
\(
\mathcal{T}_{\mu\lambda}^{(\sigma)}(x_0)=\arg\min\limits_{0\leq x< |x_0|} \widehat{L}(x,x_0),
\)
which implies that $\mathcal{T}_{\mu\lambda}^{(\sigma)}(x_0)\subseteq [0,|x_0|)$.
\end{proof}

   To further characterize this operator $\mathcal{T}_{\mu\lambda}^{(\sigma)}$, we need some properties of the derivative function $\widehat{L}'(x,x_0)$ of the objective function $\widehat{L}(x,x_0)$ on the interval $[0,+\infty)$.

   \begin{lemma}\label{ftauderivative}  Let  $x_0\in \mathbb{R}$. Then, the following statements hold.
   	\begin{itemize}
   		\item[(i)]
   		the derivative function $\widehat{L}'(x,x_0)$ is strictly convex on $[0,+\infty)$ and $\widehat{L}''(x,x_0)$ is strictly monotone increase on $[0,+\infty)$.
   		
   		\item[(ii)] If $0<\frac{\mu\lambda}{\sigma^2}\leq 1$, then $\widehat{L}'(x,x_0)$ is strictly increasing on $[0,+\infty)$. Moreover, $\widehat{L}'(x,x_0)>\frac{\lambda}{\sigma}-\frac{|x_0|}{\mu}$ for any $x>0$.
   		\item[(iii)] If $\frac{\mu\lambda}{\sigma^2}>1$, then the derivative function $\widehat{L}'(x,x_0)$ is strictly decreasing on $(0,\sigma\ln\frac{\mu\lambda}{\sigma^2 }]$, and strictly increasing on $[\sigma\ln\frac{\mu\lambda}{\sigma^2 },+\infty)$. Moreover, $\widehat{L}'(x,x_0)\geq \frac{\sigma}{\mu}(1+\ln\frac{\mu\lambda}{\sigma^2})-\frac{|x_0|}{\mu}$ for any $x>0$.
   	\end{itemize}
   \end{lemma}
   \begin{proof} By the expression of $\widehat{L}(x,x_0)$ defined as in \eqref{LDef} , after simple computation, we know that for any $x>0$,
   	\begin{equation}\label{Derivatives}
   	\begin{array}{ll}
   	&\widehat{L}'(x,x_0)=\frac{\lambda}{\sigma} e^{-\frac{x}{\sigma}}+\frac{x-|x_0|}{\mu},\\
   	&\widehat{L}''(x,x_0)=-\frac{\lambda}{\sigma^2} e^{-\frac{x}{\sigma}}+\frac{1}{\mu},\\ &\widehat{L}'''(x,x_0)=\frac{\lambda}{\sigma^3}e^{-\frac{x}{\sigma}}.
   	\end{array}
   	\end{equation}
   	Obviously, the strictly monotone increasing of $\widehat{L}''(x,x_0)$ and
   	the strict convexity of
   	$\widehat{L}'(x,x_0)$  on $[0,+\infty)$  come from the fact that
   	$\widehat{L}'''(x,x_0)>0$ for any $x>0$.
   	Notice that $\widehat{L}''(0_+,x_0)\!:=\!\lim\limits_{x\downarrow 0}\widehat{L}''(x,x_0)=\frac{1}{\mu}-\frac{\lambda}{\sigma^2}$.
   	Consequently, when $0<\frac{\mu\lambda}{\sigma^2}\le 1$, $\widehat{L}''(x,x_0)> \widehat{L}''(0_+,x_0)\geq 0$ for any $x>0$,
   	which implies that $\widehat{L}'(x,x_0)$ is strictly increasing on $(0,+\infty)$. Thus, $\widehat{L}'(x,x_0)>\widehat{L}'(0_+,x_0)\!:=\!\lim\limits_{x\downarrow 0}\widehat{L}'(x,x_0)=\frac{\lambda}{\sigma}-\frac{|x_0|}{\mu}$ for any $x>0$. Statement (ii) is obtained.
   	
   	Finally, we shall prove statement (iii) with $\frac{\mu\lambda}{\sigma^2}>1$.
   	Clearly,  from the second equation in \eqref{Derivatives}, it follows that
   	$\widehat{L}''(x,x_0)=0$ if and only if $x=\sigma\ln\frac{\mu\lambda}{\sigma^2 }>0$. Moreover, $\widehat{L}''(x,x_0)>0$ ($\widehat{L}''(x,x_0)<0$) if and only if $x>\sigma\ln\frac{\mu\lambda}{\sigma^2 }$ ($0<x<\sigma\ln\frac{\mu\lambda}{\sigma^2 }$).
   	Hence, the derivative function $\widehat{L}'(x,x_0)$ is strictly decreasing on $(0,\sigma\ln\frac{\mu\lambda}{\sigma^2 }]$, and strictly increasing on $[\sigma\ln\frac{\mu\lambda}{\sigma^2 },+\infty)$.
   	By the monotonicity of $\widehat{L}'(x,x_0)$, we know that for any $x>0$,
   $$
   		\widehat{L}'(x,x_0)\ge \widehat{L}'(\sigma\ln\frac{\mu\lambda}{\sigma^2 },x_0)=\frac{\sigma}{\mu}(1+\ln\frac{\mu\lambda}{\sigma^2})-\frac{|x_0|}{\mu}.
   $$
   	which leads to the desired result.
   \end{proof}

   \begin{theorem}\label{ProximalCase1}
   	Let $\frac{\mu\lambda}{\sigma^2}\leq 1$ and $x_0\in \mathbb{R}$, it holds that
   	\begin{align}\label{Case1}
   		\mathcal{T}_{\mu\lambda}^{(\sigma)}(x_0)\!=\!\!\left\{\begin{array}{cl}
   			\{0\}, & |x_0|\leq \frac{\mu\lambda}{\sigma},\\
   		 \{	\sigma W_0(\!-\!\frac{\mu\lambda}{\sigma^2}e^{-\frac{|x_0|}{\sigma}})\!+\!|x_0|\},\!\!\! &\mbox{otherwise}.
   		\end{array}\right.
   	\end{align}
   \end{theorem}
   \begin{proof}
   	By  Lemma \ref{ftauderivative} (ii), $\widehat{L}'(x,x_0)$  is strictly increasing on $[0,+\infty)$ and
   	\(
   \widehat{L}'(x,x_0)>\widehat{L}'(0_+,x_0)=\frac{\lambda}{\sigma}-\frac{|x_0|}{\mu}
   	\) for any $x>0$.
   	Now we will divide the arguments into two cases:
   	
   	{\bf Case 1: $0<|x_0|\le \frac{\mu\lambda}{\sigma}$}. In this case, $\widehat{L}'(x,x_0)>\widehat{L}'(0_+,x_0)\geq 0$ for all $x\in(0,\infty)$,
   	and then  $\widehat{L}(x,x_0)$ is strictly increasing on $[0,\infty)$. Hence,
   	$
   	\mathcal{T}_{\mu\lambda}^{(\sigma)}(x_0)=\{0\}.
   	$
   	
   	{\bf Case 2: $|x_0|>\frac{\mu\lambda}{\sigma}$}.   	The desired result is obtained directly by Lemma \ref{LemmaNecCondi}.
   	
   \end{proof}

    \begin{theorem}\label{ProximalCase2}
   	  Let $\frac{\mu\lambda}{\sigma^2}> 1$ and $x_0\in \mathbb{R}$, it holds that
   	\begin{align}\label{Case2}
   		\mathcal{T}_{\mu\lambda}^{(\sigma)}(x_0)=\left\{\begin{array}{cl}
   			\{0\}, & |x_0|< \sigma(1+\ln\frac{\mu\lambda}{\sigma^2}),\\
   			\arg\min\limits_{x=0,x_1}\{\widehat{L}(x,x_0)\}, & \sigma(1+\ln\frac{\mu\lambda}{\sigma^2})\leq |x_0|\leq \frac{\mu\lambda}{\sigma},\\
   			\{x_1\}, & {\rm otherwise},
   		\end{array}\right.
   	\end{align}
   	where $x_1:=\sigma W_0(-\frac{\mu\lambda}{\sigma^2}e^{-\frac{|x_0|}{\sigma}})+|x_0|$.
   \end{theorem}
  \begin{proof}
  	By  Lemma \ref{ftauderivative} (iii),  it holds that for any $x>0$
  	\begin{align}\label{LderTemp1}
  	\widehat{L}'(x,x_0)\ge \frac{1}{\mu}\Big(\sigma(1+\ln\frac{\mu\lambda}{\sigma^2})-|x_0|\Big).
  	\end{align}
  	We will divide the arguments into two cases:
  	
  	{\bf Case 1: $0< |x_0|<\sigma(1+\ln\frac{\mu\lambda}{\sigma^2})$}. In this case, $\widehat{L}'(x,x_0)>0$ for any $x\in(0,|x_0|)$ by \eqref{LderTemp1}. Hence, $\widehat{L}$ is strictly increasing on $[0,|x_0|]$, which  with  Lemma \ref{shrinkage} implies
  	$\mathcal{T}_{\mu\lambda}^{(\sigma)}(x_0)=\{0\}$

  	{\bf Case 2: $\sigma(1+\ln\frac{\mu\lambda}{\sigma^2})\leq |x_0|\leq \frac{\mu\lambda}{\sigma}$}.
  	In this case, $\widehat{L}'(0_+,x_0)\!:=\!\lim\limits_{x\downarrow 0}\widehat{L}'(x,x_0)=\frac{\lambda}{\sigma}-\frac{|x_0|}{\mu}\geq 0$ and
  	\[
  	\widehat{L}'(\sigma\ln\frac{\mu\lambda}{\sigma^2},x_0 )=\frac{1}{\mu}\big(\sigma(1+\ln\frac{\mu\lambda}{\sigma^2})-|x_0|\big)\leq 0.
  	\]
  	 Remember that $\widehat{L}'(|x_0|,x_0)=\frac{\lambda}{\sigma} e^{-\frac{|x_0|}{\sigma}}\!>\!0$.
  	By the strict convexity of $\widehat{L}'(x,x_0)$ from Lemma \ref{ftauderivative} (i),
  	there exist only two points $d_{x_0}\!<\!c_{x_0}$  such that $0\leq d_{x_0}<\sigma\ln\frac{\mu\lambda}{\sigma^2}\leq c_{x_0}<|x_0|$ and
  $\widehat{L}'(d_{x_0},x_0)=\widehat{L}'(c_{x_0},x_0)=0$. Thus, $\widehat{L}$ is increasing on $[0,d_{x_0}]$ or $[c_{x_0},|x_0|)$, and decreasing on $[d_{x_0},c_{x_0}]$. Hence,
  along with Lemma \ref{shrinkage},
  $\mathcal{T}_{\mu\lambda}^{(\sigma)}(x_0)=\arg\min\limits_{x=0,c_{x_0}}\widehat{L}(x,x_0)$. On other hand, since $\widehat{L}'(d_{x_0},x_0)=\widehat{L}'(c_{x_0},x_0)=0$,
  $d_{x_0},c_{x_0}$ are  solutions to the following equation
  \[
  \widehat{L}'(x,x_0)=\frac{1}{\mu}(x-|x_0|)+\frac{\lambda}{\sigma}e^{-\frac{x}{\sigma}}=0,
  \]
  i.e.,  \[
  \frac{x-|x_0|}{\sigma}e^{\frac{x-|x_0|}{\sigma}}=-\frac{\mu\lambda}{\sigma^2}e^{-\frac{|x_0|}{\sigma}}.
  \]
  From $|x_0|\geq \sigma (1+\ln(\frac{\mu\lambda}{\sigma^2}))$, it follows that
  \[
  -\frac{\mu\lambda}{\sigma^2}e^{-\frac{|x_0|}{\sigma}}\in [-\frac{1}{e},0).
  \]
  By the property of Lambert W function in \cite{Mezo2022}, we know that
  \begin{align*}
  	\left\{
  	\begin{array}{l}
  		c_{x_0}=x_1:=\sigma W_0(-\frac{\mu\lambda}{\sigma^2}e^{-\frac{|x_0|}{\sigma}})+|x_0|,\\
  		d_{x_0}=\sigma W_{-1}(-\frac{\mu\lambda}{\sigma^2}e^{-\frac{|x_0|}{\sigma}})+|x_0|.
  	\end{array}
  	\right.
  \end{align*}
  Thus,
  $\mathcal{T}_{\mu\lambda}^{(\sigma)}(x_0)=\arg\min\limits_{x=0,c_{x_0}}\widehat{L}(x,x_0)=\arg\min\limits_{x=0,x_1}\widehat{L}(x,x_0)$.
  		
  		{\bf Case 3: $|x_0|>\frac{\mu\lambda}{\sigma}$}. The desired result is obtained directly from Lemma \ref{LemmaNecCondi}.
  \end{proof}
  \begin{remark}
   The equation \eqref{Case2} is more detailed than the equation \eqref{ErrorFormula}, although the two equations are actually equivalent when $\frac{\mu\lambda}{\sigma^2}>1$. 
   Observe that $\mathcal{T}_{\mu\lambda}^{(\sigma)}(x_0)$ in \eqref{ErrorFormula} must be determined  when $|x_0|\ge\sigma(1+\ln\frac{\mu\lambda}{\sigma^2})$. By contrast, $\mathcal{T}_{\mu\lambda}^{(\sigma)}(x_0)$ in \eqref{Case2} must be identified only for $\sigma(1+\ln\frac{\mu\lambda}{\sigma^2})\leq |x_0|\leq \frac{\mu\lambda}{\sigma}$.
It is obvious that \eqref{Case2} takes less time to compute than \eqref{ErrorFormula}. 

  \end{remark}

  \subsection{Further discussion on $\mathcal{T}_{\mu\lambda}^{(\sigma)}$}
   The paper \cite{Malek2016} pointed that
  $|x_0|\geq \sigma(1+\ln\frac{\mu\lambda}{\sigma^2})$, it is very hard to analytically ﬁnd the condition under which $x = x_1$ or $x = 0$ is the minimizer.
  In fact, when $\frac{\mu\lambda}{\sigma^2}\leq 1$, from Theorem \ref{ProximalCase1} it follows that 
  $\widehat{L}(0,x_0)<\widehat{L}(x_1,x_0)$ for each $\sigma(1+\ln\frac{\mu\lambda}{\sigma^2})\leq |x_0|<\frac{\mu\lambda}{\sigma}$, otherwise $\widehat{L}(x_1,x_0)< \widehat{L}(0,x_0)$.
  When $\frac{\mu\lambda}{\sigma^2}>1$, from Theorem \ref{ProximalCase2} it follows that
  $\widehat{L}(x_1,x_0)<\widehat{L}(0,x_0)$
  for each $|x_0|>\frac{\mu\lambda}{\sigma}$. Moreover,  when $\frac{\mu\lambda}{\sigma^2}>1$,  we will extend the interval, see Theorem \ref{ProxTheoCase2Refined}  later, i.e.,
   $\widehat{L}(x_1,x_0)<\widehat{L}(0,x_0)$ for each $|x_0|>\min\{\frac{\mu\lambda}{\sigma},\sqrt{2\mu\lambda}\}$, and we will find a threshold value $\bar{\tau}_{\mu\lambda, \sigma}$ such that $\mathcal{T}_{\mu\lambda}^{(\sigma)}(x_0)=\{0\}$ when $|x_0|<{\bar{\tau}}_{\mu\lambda, \sigma}$ and $\mathcal{T}_{\mu\lambda}^{(\sigma)}(x_0)=\{x_1\}$ when $|x_0|>\bar{\tau}_{\mu\lambda,\sigma}$, see Theorem \ref{Coratau} later. 
  To this end, the following lemmas are key important.

\begin{lemma}\label{orderprox}(\cite[Lemma 3]{Prater2022})  Let $0\le |\tau^1|<|\tau^2|$. Then for each $\alpha\in\mathcal{T}_{\mu\lambda}^{(\sigma)}(\tau^1)$ and each $\beta\in\mathcal{T}_{\mu\lambda}^{(\sigma)}(\tau^2)$, it holds that $0\le \alpha\le \beta$.
\end{lemma}

\begin{lemma}\label{threshold}
Let  $\frac{\mu\lambda}{\sigma^2}> 1$. If there exists 
some point $\bar{\tau}_{\mu\lambda, \sigma}$ with $\sigma(1+\ln\frac{\mu\lambda}{\sigma^2})\leq \bar{\tau}_{\mu\lambda, \sigma}\le \frac{\mu\lambda}{\sigma}$  such that $\mathcal{T}_{\mu\lambda}^{(\sigma)}(\bar{\tau}_{\mu\lambda, \sigma})$   contains $0$ and a positive number, then  $\mathcal{T}_{\mu\lambda}^{(\sigma)}(x_0)=\{0\}$ for any $0<|x_0|<\bar{\tau}_{\mu\lambda, \sigma}$ and $\mathcal{T}_{\mu\lambda}^{(\sigma)}(x_0)=\{\sigma W_0(-\frac{\mu\lambda}{\sigma^2}e^{-\frac{|x_0|}{\sigma}})+|x_0|\}$ for any $|x_0|>\bar{\tau}_{\mu\lambda, \sigma}$.
\end{lemma}
\begin{proof} 
Now suppose that there exists $\bar{\tau}_{\mu\lambda, \sigma}$ with
$\sigma(1+\ln\frac{\mu\lambda}{\sigma^2})\leq \bar{\tau}_{\mu\lambda, \sigma}\le \frac{\mu\lambda}{\sigma}$ 
such that $\mathcal{T}_{\mu\lambda}^{(\sigma)}(\bar{\tau}_{\mu\lambda, \sigma})$   contains $0$ and a positive number. Then by Theorem \ref{ProximalCase2}, it holds that
\[
\mathcal{T}_{\mu\lambda}^{(\sigma)}(\bar{\tau}_{\mu\lambda, \sigma})
=\{0,W_0(-\frac{\mu\lambda}{\sigma^2}e^{-\frac{\bar{\tau}_{\mu\lambda, \sigma}}{\sigma}})+\bar{\tau}_{\mu\lambda, \sigma}\}.
\]
Since $0\!\in\!\mathcal{T}_{\mu\lambda}^{(\sigma)}(\bar{\tau}_{\mu\lambda, \sigma})$, from Lemma \ref{orderprox} it follows  that $\mathcal{T}_{\mu\lambda}^{(\sigma)}(x_0)\!=\!\{0\}$ when $0\!\le\! |x_0|\!<\!\bar{\tau}_{\mu\lambda, \sigma}$.
In the following, we  consider this case when $|x_0|\!>\!\bar{\tau}_{\mu\lambda, \sigma}$.
 Notice that
 \[
 0\!<\! W_0(-\frac{\mu\lambda}{\sigma^2}e^{-\frac{\bar{\tau}_{\mu\lambda, \sigma}}{\sigma}})+\bar{\tau}_{\mu\lambda, \sigma}\in \mathcal{T}_{\mu\lambda}^{(\sigma)}(\bar{\tau}_{\mu\lambda, \sigma}).
 \]
 We know  $0\notin \mathcal{T}_{\mu\lambda}^{(\sigma)}(x_0)$ by  Lemma \ref{orderprox}.
Associating it with Theorem \ref{ProximalCase2} yields 
$\mathcal{T}_{\mu\lambda}^{(\sigma)}(x_0)$ is a single point set for any $|x_0|\!>\!\bar{\tau}_{\mu\lambda, \sigma}$ and $\mathcal{T}_{\mu\lambda}^{(\sigma)}(x_0)\!=\!\{\sigma W_0(-\frac{\mu\lambda}{\sigma^2}e^{-\frac{|x_0|}{\sigma}})+|x_0|\}$.
\end{proof}

We will present a numerical method to compute $\bar{\tau}_{\mu\lambda, \sigma}$ in Lemma \ref{threshold}.
Let  $\frac{\mu\lambda}{\sigma^2}> 1$. 
According to  Lemma  \ref{orderprox} and Lemma  \ref{threshold}, we can 
further characterize the expression of $\mathcal{T}_{\mu\lambda}^{(\sigma)}$
on the case $\sigma(1+\ln\frac{\mu\lambda}{\sigma^2})\!\leq\! |x_0|\!\leq\! \frac{\mu\lambda}{\sigma}$ by finding the threshold value $\bar{\tau}_{\mu\lambda,\sigma}$ such that
$$
\widehat{L}(x,\bar{\tau}_{\mu\lambda,\sigma})\ge \widehat{L}(0,\bar{\tau}_{\mu\lambda,\sigma})\mbox{ for any }x>0.
$$
that is, $\frac{1}{2\mu}(x-\bar{\tau}_{\mu\lambda,\sigma})^2+\lambda(1-e^{-\frac{x}{\sigma}})\ge \frac{1}{2\mu}\bar{\tau}_{\mu\lambda,\sigma}^2$ for any $x>0$, which is equivalent to  take
\begin{equation}\label{bartau0}
\bar{\tau}_{\mu\lambda,\sigma}= \min_{x>0}H(x):=\frac12 x+\mu\lambda\frac{1-e^{-\frac{x}{\sigma}}}{x}.
\end{equation}
In the following, we will solve the optimization problem \eqref{bartau0}. To this end, 
we compute the first derivative
$$
H'(x)=\frac12+\mu\lambda\frac{(\frac{x}{\sigma}+1) e^{-\frac{x}{\sigma}}-1}{x^2},\ x>0
$$
and the second derivative
\begin{align}\label{HSecDe}
H''(x)=\mu\lambda \frac{2-(\frac{x^2}{\sigma^2}+\frac{2x}{\sigma}+2)e^{-\frac{x}{\sigma}}}{x^3},\ x>0.
\end{align}
The function $h(t):=2-(t^2+2t+2)e^{-t}> h(0)=0$ for any $t>0$ as $h'(t)=t^2e^{-t}>0$ for any $t>0$. Activating \eqref{HSecDe} with $t=\frac{x}{\sigma}$, it follows that $H''(x)>0$ for any $x>0$. Then, $H'(x)$ is increasing on $(0,+\infty)$ and
$$
H'(x)\ge \lim\limits_{x\downarrow  0}H'(x)\!:=\!H'(0_+)\!=\!\frac12(1-\frac{\mu\lambda}{\sigma^2})<0\mbox{ for any } x\!>\!0.
$$
Note that 
$$
\begin{array}{ll}
H'(\sqrt{2\mu\lambda})
&\displaystyle{=\frac12+\frac{(\frac{\sqrt{2\mu\lambda}}{\sigma}+1)e^{-\frac{\sqrt{2\mu\lambda}}{\sigma}}-1}{2} }\\
&\displaystyle{=\frac{(\frac{\sqrt{2\mu\lambda}}{\sigma}+1)e^{-\frac{\sqrt{2\mu\lambda}}{\sigma}}}{2}>0. }
\end{array}
$$
Then there exists a unique $x^*\in(0,\sqrt{2\mu\lambda})$ such that $H'(x^*)\!=\!0$.
Thus, $H(x)$ is decreasing on $(0,x^*]$ and increasing on $[x^*,+\infty)$.  The point $x^*$ is the minimizer of $H(x)$ on $(0,+\infty)$, that is, $\overline{\tau}_{\mu\lambda,\sigma}\!=\!H(x^*)$ by \eqref{bartau0}.
As $H'(x^*)\!=\!0$, we have
$$
\begin{array}{ll}
0=x^*H'(x^*)
&\displaystyle{=\frac12 x^*+\mu\lambda\frac{(\frac{x^*}{\sigma}+1) e^{-\frac{x^*}{\sigma}}-1}{x^*} }\\
&\displaystyle{=x^*+\frac{\mu\lambda}{\sigma} e^{-\frac{ x^*}{\sigma}}-H(x^*). }
\end{array}
$$
Along with $\overline{\tau}_{\mu\lambda,\sigma}=H(x^*)$, we have
\begin{equation}\label{bartau}
\bar{\tau}_{\mu\lambda,\sigma}=x^*+\frac{\mu\lambda}{\sigma} e^{-\frac{ x^*}{\sigma}}
\end{equation}
where $x^*\in(0,\sqrt{2\mu\lambda})$ is the solution of the equation $H'(x)=0$ on $(0,\infty)$. The point $x^*$ can easily be found by the bisection method. 
For some choices $\mu,\lambda,\sigma>0$ such that $\frac{\mu\lambda}{\sigma^2}>1$, the corresponding values $x^*$ and $\bar{\tau}_{\mu\lambda,\sigma}$ are listed in TABLE \ref{Table1}. Numerical experiments demonstrate that for a fix $\mu\lambda>0$, $\bar{\tau}_{\mu\lambda,\sigma}$ is increasing and tends to $\sqrt{2\mu\lambda}$ as $\sigma$ tends to $0$. 
\begin{table}[hbt]
\centering
\caption{The values $x^*$ and $\bar{\tau}_{\mu\lambda,\sigma}$ for some $\mu,\lambda,\sigma>0$ such that $\frac{\mu\lambda}{\sigma^2}>1$.}
\resizebox{\linewidth}{!}{
\label{Table1}
\begin{tabular}{|c|c|c|c|c|c|c|c|} \hline
&$\sigma$ & 1.4 & 1 & 0.5 & 0.3 & 0.2 & 0.1 \\ 
$\mu\lambda=2$ & $x^*$ &0.04247947& 1.09157888 &1.88725512 & 1.98992887 & 1.9994994 & 1.99999996 \\ 
& $\bar{\tau}_{\mu\lambda,\sigma}$ &1.42835552 & 1.76295101 & 1.97904843 & 1.99870274& 1.99995454 & 2 \\ \hline
&$\sigma$ &0.99 & 0.9 & 0.5 & 0.3 & 0.2 & 0.1 \\ 
$\mu\lambda=1$ & $x^*$ &0.02988714 & 0.28837712& 1.16132153 & 1.3730464 & 1.40925117 & 1.41420584\\
& $\bar{\tau}_{\mu\lambda,\sigma}$ &1.00994987 & 1.09487137 & 1.3573499 & 1.40733821 & 1.41360448 & 1.41421305 \\ \hline
&$\sigma$ &0.49 & 0.3& 0.2 & 0.1 & 0.05 & 0.02\\
$\mu\lambda=\frac{1}{4}$ & $x^*$ &  0.02977356 & 0.49609826 & 0.64499062& 0.70462559 & 0.70710291& 0.70710678\\
& $\bar{\tau}_{\mu\lambda,\sigma}$ &0.5098995 & 0.65555503 & 0.69468768& 0.70680224& 0.70710652& 0.70710678\\ \hline
\end{tabular}}
\end{table}

With the help of Corollary \ref{threshold}, the following result holds.

\begin{theorem}\label{Coratau} Let $\frac{\mu\lambda}{\sigma^2}>1$ and $x_0\in \bR$. Then
\begin{equation}\label{ProximalThreshold}
\mathcal{T}_{\mu\lambda}^{(\sigma)}(x_0)=\left\{\begin{array}{ll}
\{0\}, & |x_0|< \bar{\tau}_{\mu\lambda,\sigma},\\
\{0,x_1\},& |x_0|=\bar{\tau}_{\mu\lambda,\sigma},\\
\{x_1\},& \mbox{otherwise},
\end{array}
\right.
\end{equation}
where the threshold $\bar{\tau}_{\mu\lambda,\sigma}$ is obtained by Equation \eqref{bartau}
and $x_1:=\sigma W_0(-\frac{\mu\lambda}{\sigma^2}e^{-\frac{|x_0|}{\sigma}})+|x_0|$.
\end{theorem}

We remark that the threshold $\bar{\tau}_{\mu\lambda,\sigma}$ in Theorem \ref{Coratau}
does not have a closed form and is obtained by solving an inverse problem. Motivated by the threshold value $\bar{\tau}_{\mu\lambda,\sigma}$, we will tight the interval in case 2 in \eqref{Case2}. The following corollary can be obtained by activating Theorem \ref{ProximalCase2} with $x_0=\sqrt{2\mu\lambda}$.
\begin{corollary}\label{SpExample}
   Denote $x_1\!:=\!\sigma W_0(-\frac{\mu\lambda}{\sigma^2}e^{-\frac{\sqrt{2\mu\lambda}}{\sigma}})+\sqrt{2\mu\lambda}$. Then
      if $\frac{\mu\lambda}{\sigma^2}\!>\!2$, one has
      \(
    \mathcal{T}_{\mu\lambda}^{(\sigma)}(\sqrt{2\mu\lambda})\!=\!\{x_1\}
    \) 
    and
     \[
     x_1\in(\max\{\sigma\ln\frac{\mu\lambda}{\sigma^2},\sqrt{2\mu\lambda}-2\sigma\},\sqrt{2\mu\lambda}).
     \]
\end{corollary}
\begin{proof}
 Define $h(t):=t^2 e^{2-\sqrt{2}t}-2$ for any $t\!\in\! \mathbb{R}$. Then for any $t>\sqrt{2}$, it holds that
 \[
 h'(t)\!=\!\sqrt{2}t(\sqrt{2}-t)e^{2-\sqrt{2}t}\!<\!0.
 \]
 Hence, 
$h(t)<h(\sqrt{2})=0$ for any  $t>\sqrt{2}$.
On the other hand, taking $x=\sqrt{2\mu\lambda}-2\sigma$ and 
$x_0=\sqrt{2\mu\lambda}$, it follows from the first equation in \eqref{Derivatives} that
\begin{align}\label{Fderivative}
\widehat{L}'(\sqrt{2\mu\lambda}-2\sigma,\sqrt{2\mu\lambda})
=\frac{\sigma}{\mu}(\frac{\mu\lambda}{ \sigma^2}e^{2-\frac{\sqrt{2\mu\lambda}}{\sigma}}-2) <0.
\end{align}
where the inequality is by activating the fact $h(t)\!<\!h(\sqrt{2})\!=\!0$ for any $t>\sqrt{2}$ by taking $\overline{t}\!=\!\frac{\sqrt{\mu\lambda}}{\sigma}\!>\!\sqrt{2}$.
Again since $\frac{\mu\lambda}{\sigma^2}\!>\!2$,  it is easily to check that
$x_0:=\sqrt{2\mu\lambda}<\frac{\mu\lambda}{\sigma}$ and
$x_0>\sigma(1+\ln\frac{\mu\lambda}{\sigma^2})$ where is by the facts
the function $g(t)=1+2ln t-\sqrt{2}t$ is decrease on $[\sqrt{2},\infty)$ and 
$1+ln2< 2$ from Lemma \ref{ThresholdLemma}.
Thus $\sigma(1+\ln\frac{\mu\lambda}{\sigma^2})<x_0<\frac{\mu\lambda}{\sigma}$.
From the proof of the case 2 in Theorem \ref{ProximalCase2} and \eqref{Fderivative}, it follows that $\sigma\ln\frac{\mu\lambda}{\sigma^2}\leq x_1<|x_0|$ and $\sqrt{2\mu\lambda}-2\sigma<x_1$.
Hence, 
\[
x_1\in(\max\{\sigma\ln\frac{\mu\lambda}{\sigma^2},\sqrt{2\mu\lambda}-2\sigma\},\sqrt{2\mu\lambda}).
\]
Notice
\begin{align*}
&\widehat{L}(x_1,\sqrt{2\mu\lambda})-\widehat{L}(0,\sqrt{2\mu\lambda}) \\
&=\lambda+\frac{\sigma(x_1-\sqrt{2\mu\lambda})}{\mu}+\frac{1}{2\mu}x_1^2-x_1\frac{\sqrt{2\mu\lambda}}{\mu}\nonumber\\
&=\frac{1}{2\mu}x_1^2-(\frac{\sqrt{2\mu \lambda}-\sigma}{\mu})x_1+\lambda-\frac{\sigma\sqrt{2\mu\lambda}}{\mu}\\
&=\frac{1}{2\mu}(x_1-\sqrt{2\mu\lambda})
(x_1-\sqrt{2\mu\lambda}+2\sigma)<0
\end{align*}
where the first equality is by \eqref{LDef} and \eqref{WSolution}.
Along with Theorem \ref{ProximalCase2}, we know that and $\mathcal{T}_{\mu\lambda}^{(\sigma)}(\sqrt{2\mu\lambda})=\{x_1\}$.
\end{proof}

\begin{remark}
Let   $\frac{\mu\lambda}{\sigma^2}\!>\!2$.
 $x_1\!:=\!\sigma W_0(-\frac{\mu\lambda}{\sigma^2}e^{-\frac{\sqrt{2\mu\lambda}}{\sigma}})+\sqrt{2\mu\lambda}$ in  Corollary \ref{SpExample} satisfies
$x_1\in (\sigma(1+\ln\frac{\mu\lambda}{\sigma^2}),\sqrt{2\mu\lambda})$ when $\sqrt{2}<\frac{\sqrt{\mu\lambda}}{\sigma}<\iota$, otherwise
$x_1\in (\sqrt{2\mu\lambda}-2\sigma,\sqrt{2\mu\lambda})$
where $\iota\approx 2.93868$ is a root of the equation $\sqrt{2}t-2\ln t-2=0$. 
In addition, fix any $\mu\lambda>0$, then
$x_1\to \sqrt{2\mu\lambda}$ when $\sigma \to 0$.
\end{remark}
\begin{proof}
 Let $h(t):=\sqrt{2}t-2-2\ln t$ for any $t>\sqrt{2}$.
It follows that $h'(t)=\sqrt{2}-\frac{2}{t}>0$ for any $t>\sqrt{2}$. Then,
$h(t)$ is strictly increasing on $[\sqrt{2},+\infty)$. Notice that $h(\sqrt{2})=-\ln 2<0$ and $h(4\sqrt{2})=6-5\ln 2>0$. Denote by $\iota\in(\sqrt{2},4\sqrt{2})$ the only root of the equation $h(t)=0$ on $(\sqrt{2},+\infty)$. Therefore, $h(t)\ge h(\iota)=0$ for any $t\ge\iota$, by taking $t=\frac{\sqrt{\mu\lambda}}{\sigma}$, which 
implies that $\sqrt{2\mu\lambda}-2\sigma\ge\sigma\ln\frac{\mu\lambda}{\sigma^2}$ when $\frac{\sqrt{\mu\lambda}}{\sigma}\ge\iota$.
Clearly, fix $\mu\lambda>0$, we have $x_1\to\sqrt{2\mu\lambda}$ whenever $\sigma\to 0$. 
\end{proof}

With the help of Corollary \ref{SpExample} and Theorem \ref{Coratau}, a refined conclusion for Theorem \ref{ProximalCase2} can be obtained. 
\begin{theorem}\label{ProxTheoCase2Refined} 
 Let $\frac{\mu\lambda}{\sigma^2}>1$ and $x_0\in \bR$. Then
{\footnotesize\begin{align}\label{RefineProx}
\mathcal{T}_{\mu\lambda}^{(\sigma)}(x_0)\!\!=\!\!\left\{\begin{array}{ll}
\{0\}, & |x_0|< \sigma(1+\ln \frac{\mu\lambda}{\sigma^2}),\\
\arg\min\limits_{x=0,x_1}\{\widehat{L}(x,x_0)\}\!,\! \!\!\!& \sigma(1\!+\!\ln \frac{\mu\lambda}{\sigma^2})\leq |x_0|\!\le \!\min\{\frac{\mu\lambda}{\sigma},\!\!\sqrt{2\mu\lambda}\},\\
x_1,& |x_0|>\min\{\frac{\mu\lambda}{\sigma},\sqrt{2\mu\lambda}\},
\end{array}
\right.
\end{align}}
where   $x_1:=\sigma W_0(-\frac{\mu\lambda}{\sigma^2}e^{-\frac{|x_0|}{\sigma}})+|x_0|$.
\end{theorem} 
\begin{proof} It is easily checked that for any $\frac{\mu\lambda}{\sigma^2}\leq 2$,
\[
[\sigma(1+\ln \frac{\mu\lambda}{\sigma^2}),\frac{\mu\lambda}{\sigma}]= [\sigma(1+\ln \frac{\mu\lambda}{\sigma^2}),\min\{\frac{\mu\lambda}{\sigma},\sqrt{2\mu\lambda}\}].
\]
By Theorem \ref{ProximalCase2}, it is sufficient to prove the equation \eqref{RefineProx}  holds for any $\frac{\mu\lambda}{\sigma^2}>2$. 
Assume that $\frac{\mu\lambda}{\sigma^2}>2$ and $|x_0|>\sigma(1+\ln \frac{\mu\lambda}{\sigma^2})$. Then $\mathcal{T}_{\mu\lambda}^{(\sigma)}(\sqrt{2\mu\lambda})=\{x_1\}$ and $x_1<\sqrt{2\mu\lambda}$ follow from
Corollary \ref{SpExample}. Again by Theorem \ref{Coratau}, we know that 
$\bar{\tau}_{\mu\lambda,\sigma}\leq \sqrt{2\mu\lambda}$, $\bar{\tau}_{\mu\lambda,\sigma}$ defined as in Theorem \ref{Coratau}.
Hence, the desired result can be obtained.
\end{proof}

The proximal operator of PiE function computed by Equation \eqref{ProxTmulambdasigma} and 
Theorems  \ref{ProximalCase1}
and  \ref{ProxTheoCase2Refined}, is illustrated in Fig. \ref{PiEProximal}.

\begin{figure}[!tbp]
\begin{center}
\includegraphics[scale=0.6]{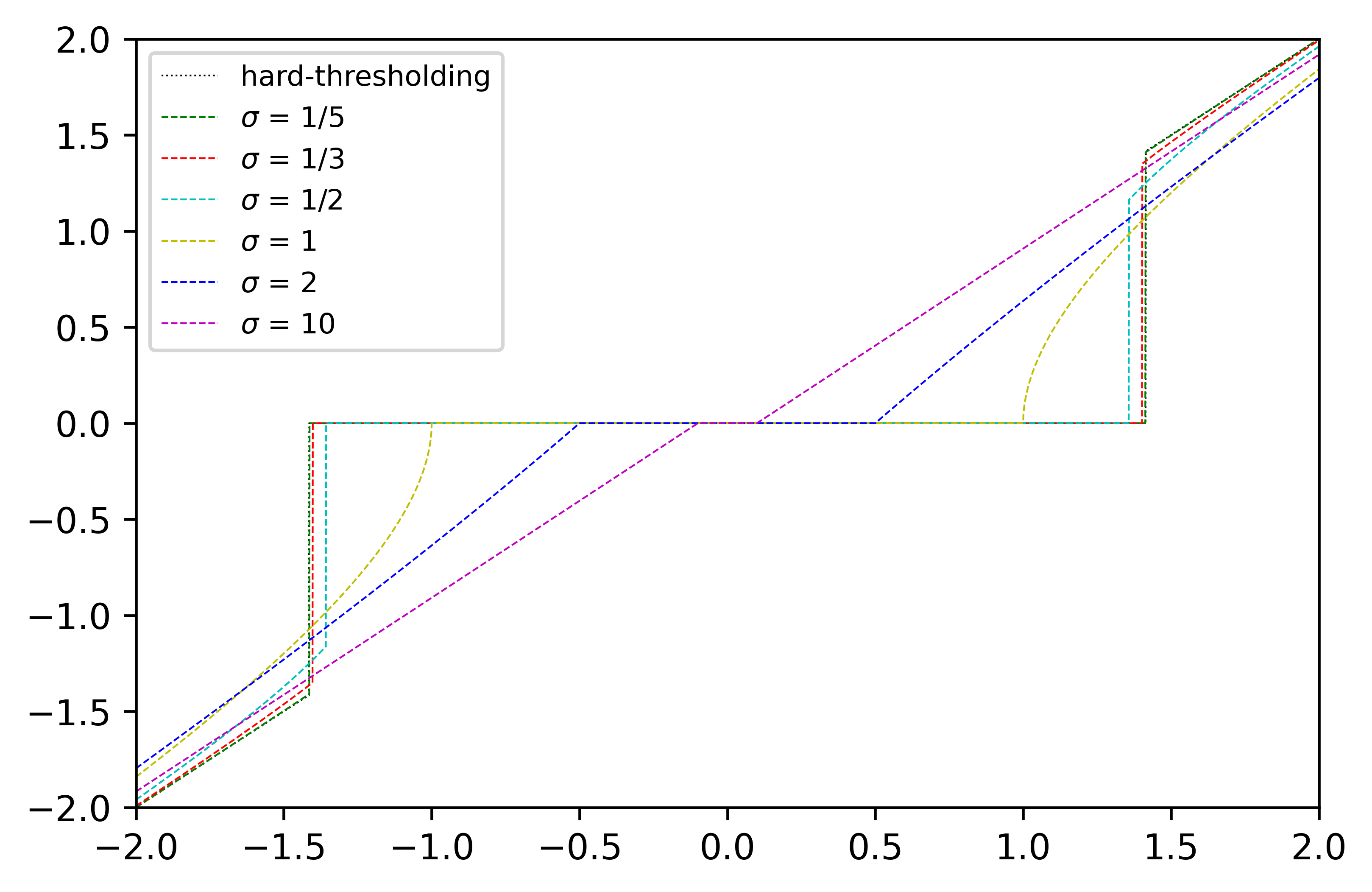}
\caption{Let $\mu\lambda=1$. The proximal operator $\prox^{(\sigma)}_{\mu\lambda }$ (dashed) of PiE $f_{\sigma}$ with $\sigma=\{1/5,1/3,1/2,1,2,10\}$  and the hard-thresholding operator (dashdot).}
\label{PiEProximal}
\end{center}
\end{figure}

At the end of this section, when $\frac{\mu\lambda}{\sigma^2}>1$, we compare the efficiency of computing the proximal operator $\prox^{(\sigma)}_{\mu\lambda f }$ for PiE by three formulas: \eqref{ErrorFormula}, \eqref{RefineProx} in Theorem \ref{ProxTheoCase2Refined}, and \eqref{ProximalThreshold} in Theorem \ref{Coratau}.
We choose one million equally distributed points on $[0,10]$. The computation time for those three methods is listed in TABLE \ref{Table}. The best results are highlighted in boldface. The proximal operator computed by \eqref{ProximalThreshold} requires approximately $60\%$ as much time as \eqref{ErrorFormula}. 

\begin{table}[!hbt]
\centering
\caption{Computation time for three formulas.}
\label{Table}
\resizebox{\linewidth}{!}{
\begin{tabular}{|c|c|c|c|c|} \hline
Formulas &  \eqref{ErrorFormula}  & \eqref{RefineProx} & \eqref{ProximalThreshold}\\ \hline
 $\mu=1$, $\lambda=1$, $\sigma=1/5$&0.3281s& 0.2656s& {\bf 0.1712}s\\
 $\mu=1$, $\lambda=1/2$, $\sigma=1/2$ &0.3594s&0.3125s& {\bf 0.2188}s  \\
 $\mu=1$, $\lambda=1/10$, $\sigma=1/5$ &0.3125s&0.2500s& {\bf 0.1875}s  \\
 $\mu=1/5$, $\lambda=1/10$, $\sigma=1/10$  &0.5313s&0.3594s&{\bf 0.2813}s \\
\hline
\end{tabular}}
\end{table}

\section{An application in compressed sensing}\label{Section:CS}

Suppose that $A$ is a given $m\times n$ measurement matrix with $m<n$ and $\bb\in\bR^m$ is the observation vector.
Reconstructing a sparse signal from linear measurements is a current issue in compressed sensing, i.e.,
$
\bb=A\bx.
$
To find the sparse signal $\bx\in\bR^n$, 
the following  optimization model is often used to
\begin{equation}\label{CSOptimization}
\min_{\bx\in\bR^n} \frac12\|A\bx-\bb\|_2^2+\lambda P(\bx)
\end{equation}
where $\lambda>0$ is a regularization parameter and $P(\bx)$ is a penalty term, which  reflects our prior knowledge about the signal to be recovered.
In this paper, we consider $P(\bx)$ be the PiE on $\bR^n$, i.e., for any $\bx\in\bR^n$
\begin{equation}\label{Pfsigmand}
P(\bx)=P_{\sigma}(\bx):=\sum_{i=1}^n f_{\sigma}(x_i)=\sum_{i=1}^n (1-e^{-|x_i|/\sigma}).
\end{equation}
where $f_{\sigma}$ is defined as in \eqref{PiE}.

\subsection{Iterative Shrinkage and Thresholding Algorithm}\label{Subsection:ISTA}

In this subsection, we are devoted to applying ISTA \cite{Bayram2015} to  the problem \eqref{CSOptimization} with the PiE penalized function, i.e.,  
\begin{equation}\label{PCSOptimization}
\min_{\bx\in\bR^n} \frac12\|A\bx-\bb\|_2^2+\lambda P_{\sigma}(\bx)
\end{equation}
where $P_{\sigma}$ is given by \eqref{Pfsigmand}.
Given the current iterate point $\bx^l$, the first term in \eqref{PCSOptimization} can be bounded above by the second order Taylor expansion at $\bx^l$. Thus, the next iterate point $\bx^{l+1}$ can be given by
$$
\begin{array}{ll}
&\bx^{l+1}
\!:=\!\arg\min\limits_{\bx\in\bR^n} \frac12\|A\bx^l-\bb\|_2^2\!+\!(\bx-\bx^{l})^{\top}(A^{\top}(A\bx^l-\bb))\\ &\qquad\qquad+\frac{1}{2\mu}\|\bx-\bx^{l}\|^2_2 +\lambda P_{\sigma}(\bx)
\end{array}
$$
where  $\mu$ is a constant parameter. 
After some simple algebraic manipulation and cancellation of constant terms, the above iteration expression  can be rewritten as
$$
\bx^{l+1}=\arg\min_{\bx\in\bR^n} \lambda P_{\sigma}(\bx)+ \frac{1}{2\mu}\Big\|\bx-\big(\bx^l-\mu(A^{\top}(A\bx^l-\bb)\big) \Big\|_2^2,
$$
which by the definition of the proximal operator yields
\begin{equation}\label{ISTAIteration}
\bx^{l+1}=\prox_{\mu\lambda P_{\sigma}}\Big((I_n-\mu A^{\top}A)\bx^l+\mu A^{\top}\bb \Big).
\end{equation}
The iteration expression \eqref{ISTAIteration} was well known as ISTA or forward-backward splitting algorithm in \cite{Bayram2015,Beck2017,Xu2023}. The convergence of ISTA 
 has been well studied in \cite{Bayram2015}, when the penalty term $P(x)$ is
 a $\rho$ weakly convex penalty function, i.e.,  there exist $\rho\geq 0$ such that
$P(\bx)+\frac{\rho}{2}\|\bx\|_2^2$ is convex \cite[Definition 1]{Bayram2015}.
Obviously,  the function $\lambda P_{\sigma}$ is $\rho$-weakly convex function for any $\rho\ge\frac{\lambda}{\sigma^2}$ \cite[Table 1]{Chen2014}. Denote the minimum and maximum eigenvalue of $A^{\top}A$ by $\nu_{\min}(A^{\top}A)$ and $\nu_{\max}(A^{\top}A)$, respectively.  
Associated with  \cite[Proposition 4]{Bayram2015}, when
 $\frac{\lambda}{\sigma^2}\le \rho\le\nu_{\min}(A^{\top}A)$, the optimal solution set of \eqref{PCSOptimization} is non-empty. Moreover, when $0<\mu<\frac{2}{\nu_{\max}(A^{\top}A) +\rho}$, the sequence $\{\bx^l:l\in\bN\}$ in \eqref{ISTAIteration} converges to a minimizer of  \eqref{PCSOptimization}.
With \eqref{ProxDef0}, \eqref{ProxTmulambdasigma}, and the separability of  $P_{\sigma}(\bx)$,
the proximal operator $\prox_{\mu\lambda P_{\sigma}}$ in \eqref{ISTAIteration} takes the following form:
\begin{equation}\label{ProximalPiEMulti}
\begin{array}{l}
\prox_{\mu\lambda P_{\sigma}}(\bx)
\!=\!\big(\prox^{(\sigma)}_{\mu\lambda }(x_i): i=1,2,\dots,n\big)^{\top} \\
\qquad\qquad\quad\!=\!\big({\rm sign}(x_i)\mathcal{T}_{\mu\lambda}^{(\sigma)}(x_i): i=1,2,\dots,n\big)^{\top}
\end{array}
\end{equation}
where $\mathcal{T}_{\mu\lambda}^{(\sigma)}$ is given by Theorem \ref{ProximalCase1} if $\frac{\mu\lambda}{\sigma^2}\le 1$; otherwise  $\mathcal{T}_{\mu\lambda}^{(\sigma)}$ is given by Theorem \ref{Coratau}.
Based on the above arguments, the ISTA for solving problem \eqref{PCSOptimization} is described in Algorithm \ref{AlgoISTA}. For simplicity, we denote the maximum step size for ISTA by
\begin{equation}\label{mumax}
\mu_{\max}:=\frac{2}{\nu_{\max}(A^{\top}A)+\frac{\lambda}{\sigma^2}}.
\end{equation}

\begin{algorithm}
  \caption{ISTA for the problem \eqref{PCSOptimization}}
  \label{AlgoISTA}
  {\bf Input}: Given $\bx^{0}\in\bR^n$, $\varepsilon>0$, {\it maxiter}, $\lambda>0$, $\sigma>0$, $0<\mu<\mu_{\max}$, where $\mu_{\max}$ is given by \eqref{mumax}
  \begin{itemize}
  \setlength{\itemsep}{-0.3ex}
  \item[1.]  {\bf while } $e>\varepsilon$ and $l<\mbox{maxiter}$ {\bf do}
  \item[2.] \quad $l=l+1$
  \item[3.] \quad $\bx^{l+1}=\prox_{\mu\lambda P_{\sigma}}\Big((I_n-\mu A^{\top}A)\bx^l+\mu A^{\top}\bb\Big)$, 
  \item[] \qquad where $\prox_{\mu\lambda P_{\sigma}}$ is given by \eqref{ProximalPiEMulti}
 \item[4.]  \quad $e=\frac{\|\bx^{l+1}-\bx^l\|_2}{1+\|\bx^l\|_2}$
  \item[5.] {\bf end while}
  \end{itemize}
  {\bf Output}: $\bx^{l}$
  \end{algorithm}

\subsection{Numerical experiments}\label{Subsection:Experiments}
In  this subsection, the numerical performance of ISTA given as in Algorithm \ref{AlgoISTA} to solve the optimation \eqref{PCSOptimization} is studied. 
Throughout this section, 
fix $m=128$ and $n=256$ for the sensing matrix $A\in\bR^{m\times n}$. 
Sparsity levels $k\in\{4,8,\dots,60\}$ of a signal $\bx\in\bR^n$ are taken to represent
simple to difficult reconstruction problems. The maximum number of iterations is set to be $\mbox{maxiter}=3000$ and  the tolerance $\varepsilon=10^{-5}$ for ISTA.
For sensing matrix $A$, its columns
are normalized with respect to the $\ell_2$ norm and its coherence is the maximum absolute value of
the cross-correlations between its columns. Incoherent or coherent sensing matrices are produced in numerical experiments.
A random standard Gaussian matrix $A$ was commonly used in the incoherence case.
In the highly coherent case, we consider a randomly oversampled partial discrete cosine transform (DCT) matrix $A$ defined by 
$$
A_{ij}=\frac{1}{\sqrt{m}}\cos\frac{2(j-1)\pi \xi_i}{F},\ i=1,2,\dots,m, j=1,2,\dots,n
$$
where each frequency $\xi_i$ generated from a uniform distribution on $[0,1]$, and $F\in\bN$ is the refinement factor. A normalized DCT matrix $A$ with $F=3$ and $F=10$ respectively corresponds to matrices with moderate and high degrees of coherence. The average mutual coherence of $100$ random Gaussian matrice, DCT matrices with $F=3$ and $F=10$, is $0.37$ (with the standard variance $0.02$), $0.68$ (with the standard variance $0.04$), and  $0.998$ (with the standard variance $0.0016$), respectively.
 The influence of three parameters (the regularization parameter $\lambda$, the shape parameter $\sigma$, and the step size $\mu$) in Algorithm \ref{AlgoISTA} will be discussed for those three types of sensing matrices.

Given a sparse signal $\bx\in\bR^n$ with sparsity $k=60$ and the amplitudes of its nonzero coefficients are uniformly distributed on $[-5,5]$, and the sensing matrix $A$ is an incoherent Gaussian matrix, the reconstruction signal $\hat{\bx}$ by Algorithm \ref{AlgoISTA} to the optimal problem \eqref{PCSOptimization} with $\bb=A\bx$ is presented in Fig. \ref{ExampleCS}.
\begin{figure}[tbp]
\begin{center}
\includegraphics[scale=0.6]{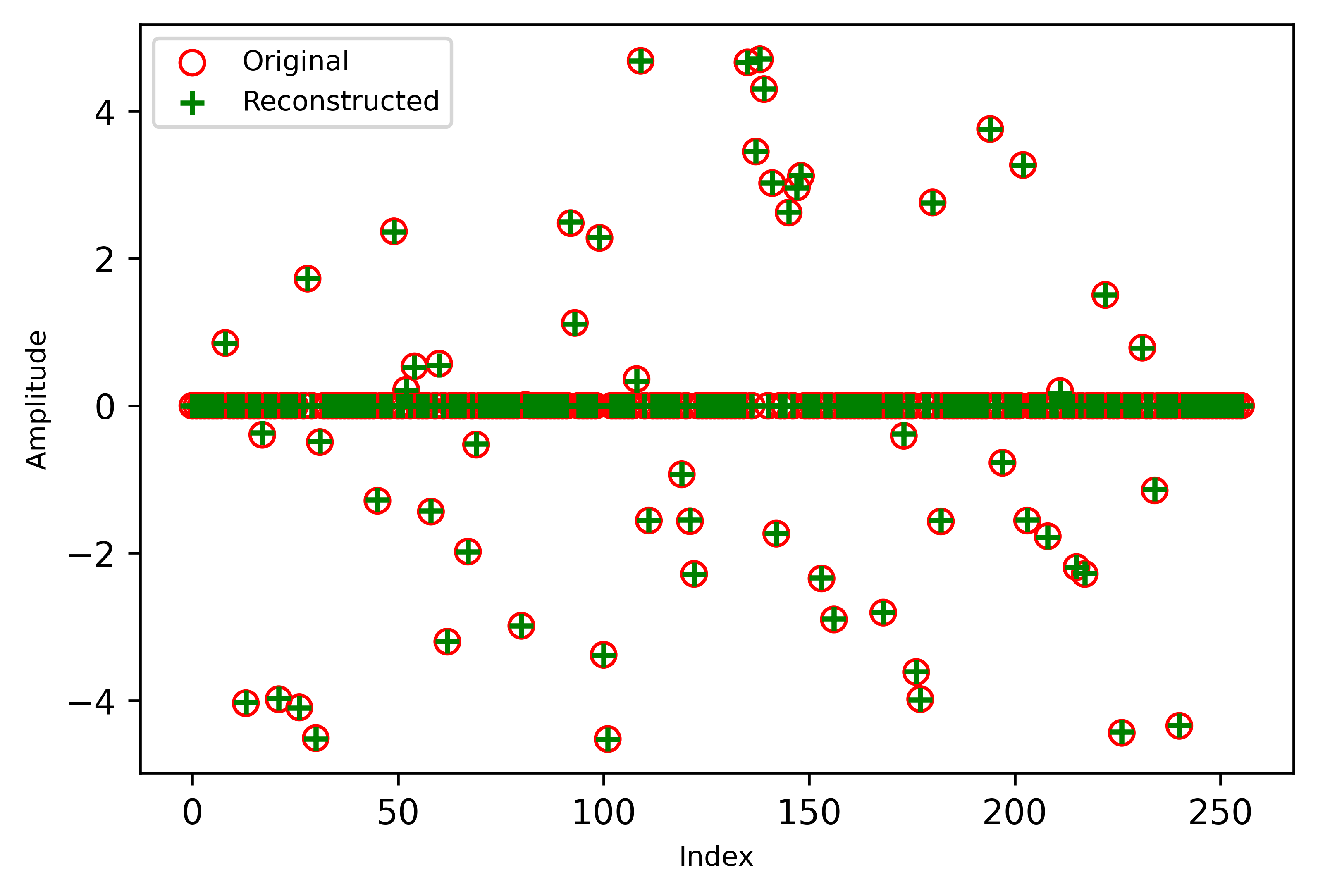}
\caption{An example of reconstruction  by ISTA via PiE with $\lambda=0.01$, $\mu=0.99\mu_{\max}$, $\sigma=0.5$, and $k=60$. The root relative squared error is $0.0042$.}
\label{ExampleCS}
\end{center}
\end{figure}
The reconstruction $\hat{\bx}$  is said to be {\it successful} if the
root relative squared error $\|\hat{\bx}-\bx\|_2/\|\bx\|_2$ is below $0.01$, and the success rate is the percentage of successful recovery after executing $100$ independent trials. 
The following numerical experiments will derive the impact of  parameters on success rate. First,  fixing $\lambda=0.01$, the success rate of ISTA using different $\sigma\in\{1/100,1/50,1/20,1/10,1/5,1/2,1,2\}$ is given in Fig. \ref{Fig5} with $\mu=0.5\mu_{\max}$ and in Fig. \ref{Fig6} with $\mu=0.99\mu_{\max}$.
The best overall success rate of ISTA can be obtained when $\sigma=1/2$. 
Second, fix $\sigma=1/2$ and let $\mu\in\{0.5\mu_{\max},0.99\mu_{\max}\}$.
Fig. \ref{Fig7} and Fig. \ref{Fig8}  show that the best regularization parameter $\lambda$ is $1/100$ among the values $\{1/10,1/20,1/100,1/200,1/1000\}$. Lastly,  fixing $\sigma\!=\!1/2$ and $\lambda\!=\!1/100$. Fig. \ref{Fig9} demonstrates that the larger the step size $\mu$, the better the success rate of ISTA. The best success rate of ISTA achieves when $\mu\!=\!0.99\mu_{\max}$ for three types of aforementioned sensing matrices. The paper \cite{Bayram2015} demonstrated numerically that larger step size $\mu$ leads to faster convergence of ISTA.

\begin{figure*}[htbp]
\centering
\subfloat[]{\includegraphics[width=2.3in]{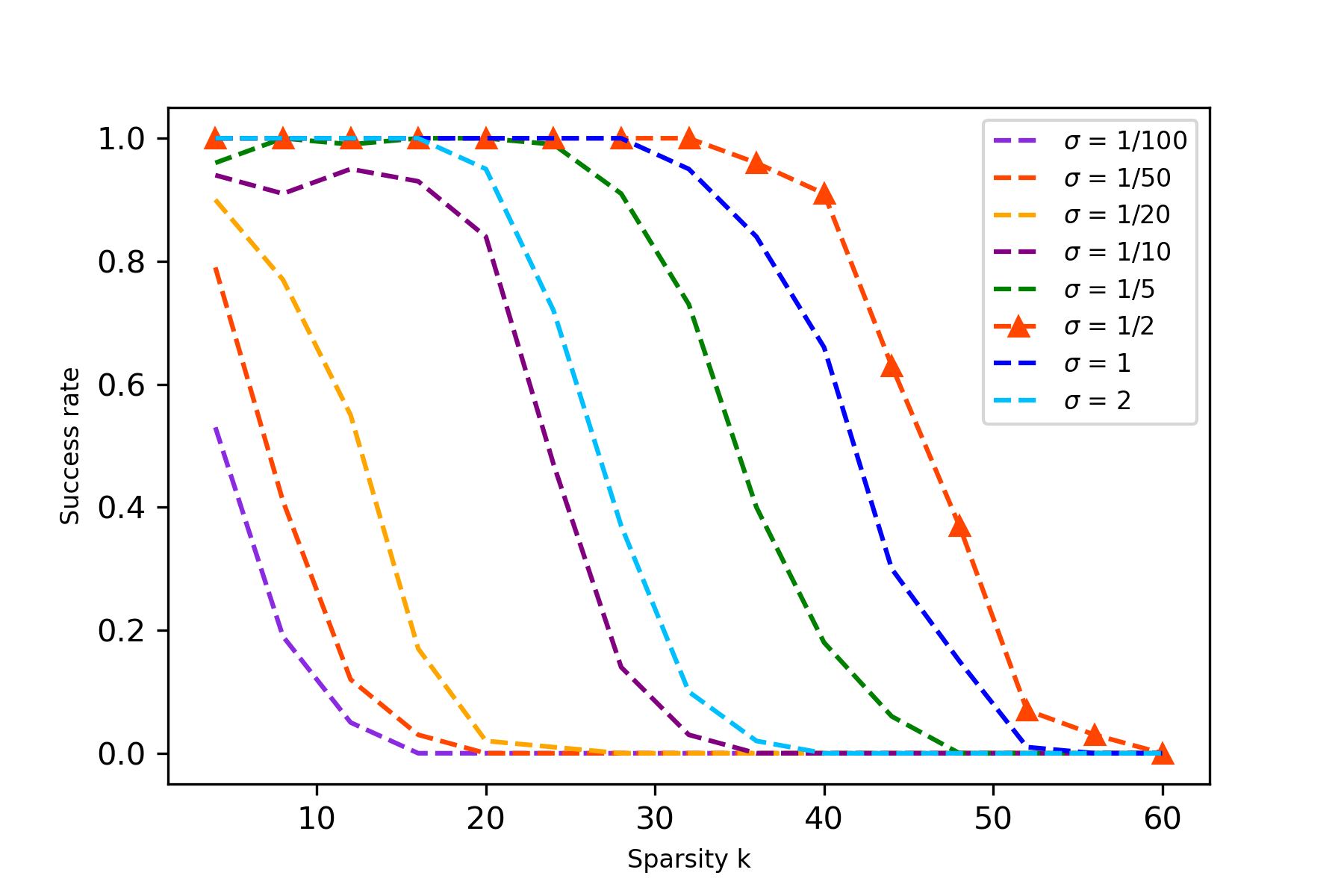}}
\hfil
\subfloat[]{\includegraphics[width=2.3in]{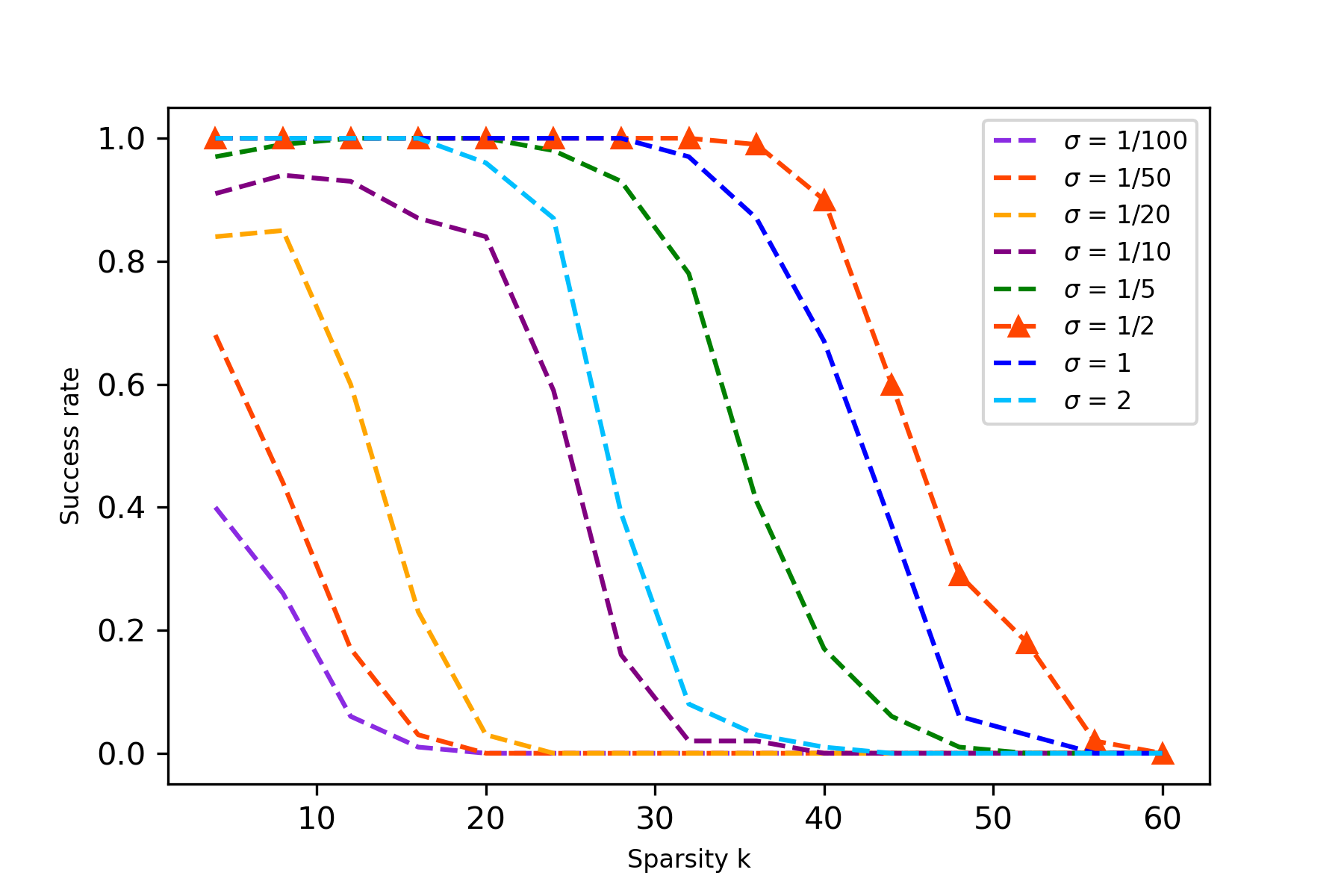}}
\hfil
\subfloat[]{\includegraphics[width=2.3in]{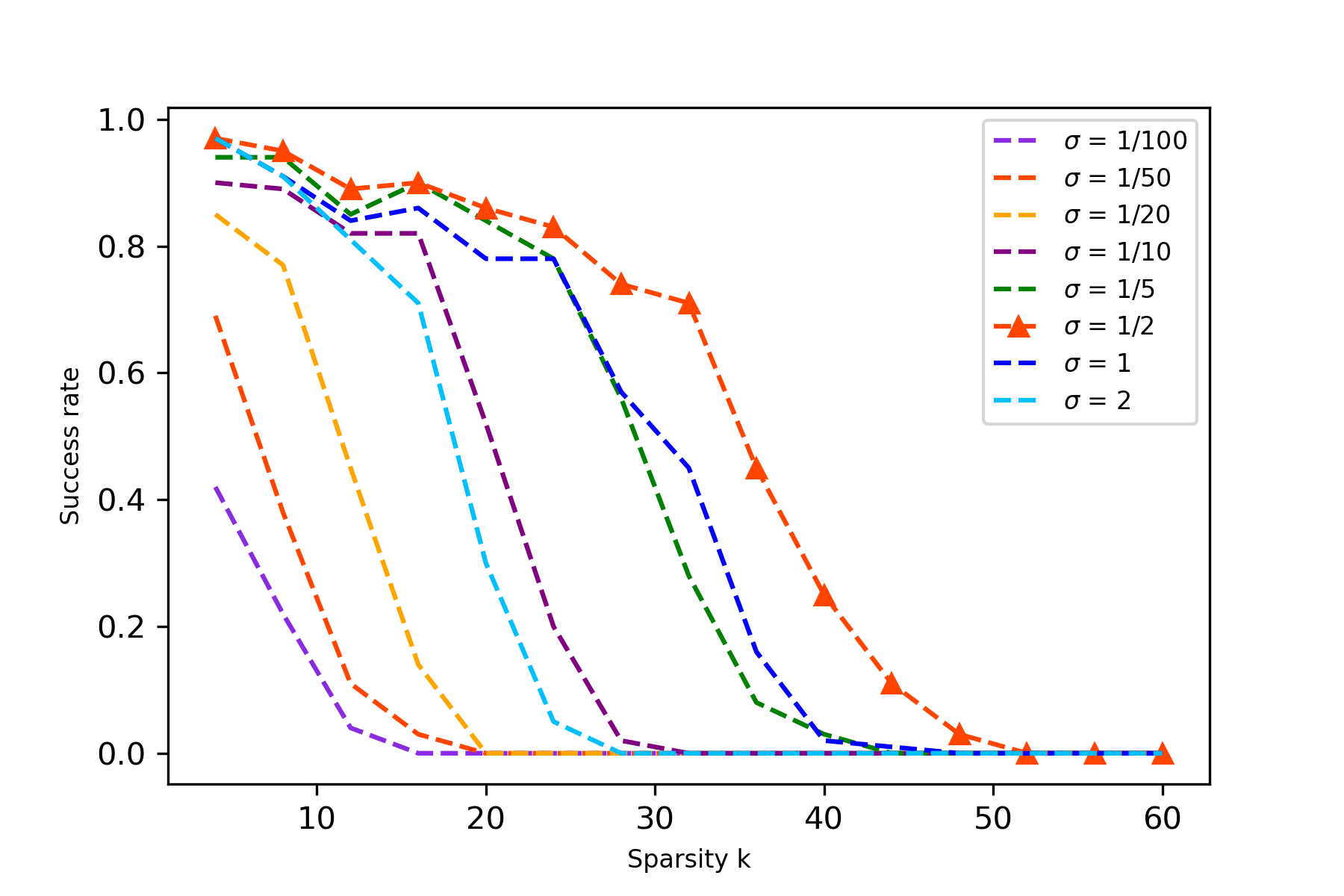}}
\hfil
\caption{Success rate of ISTA using different $\sigma$ with $\mu = 0.5\mu_{max}$ and $\lambda=1/100$ via PiE. (a) Gaussian matrix; (b) DCT matrix with $F=3$; (c) DCT matrix with $F=10$.}
\label{Fig5}
\end{figure*}

\begin{figure*}[htbp]
\centering
\subfloat[]{\includegraphics[width=2.3in]{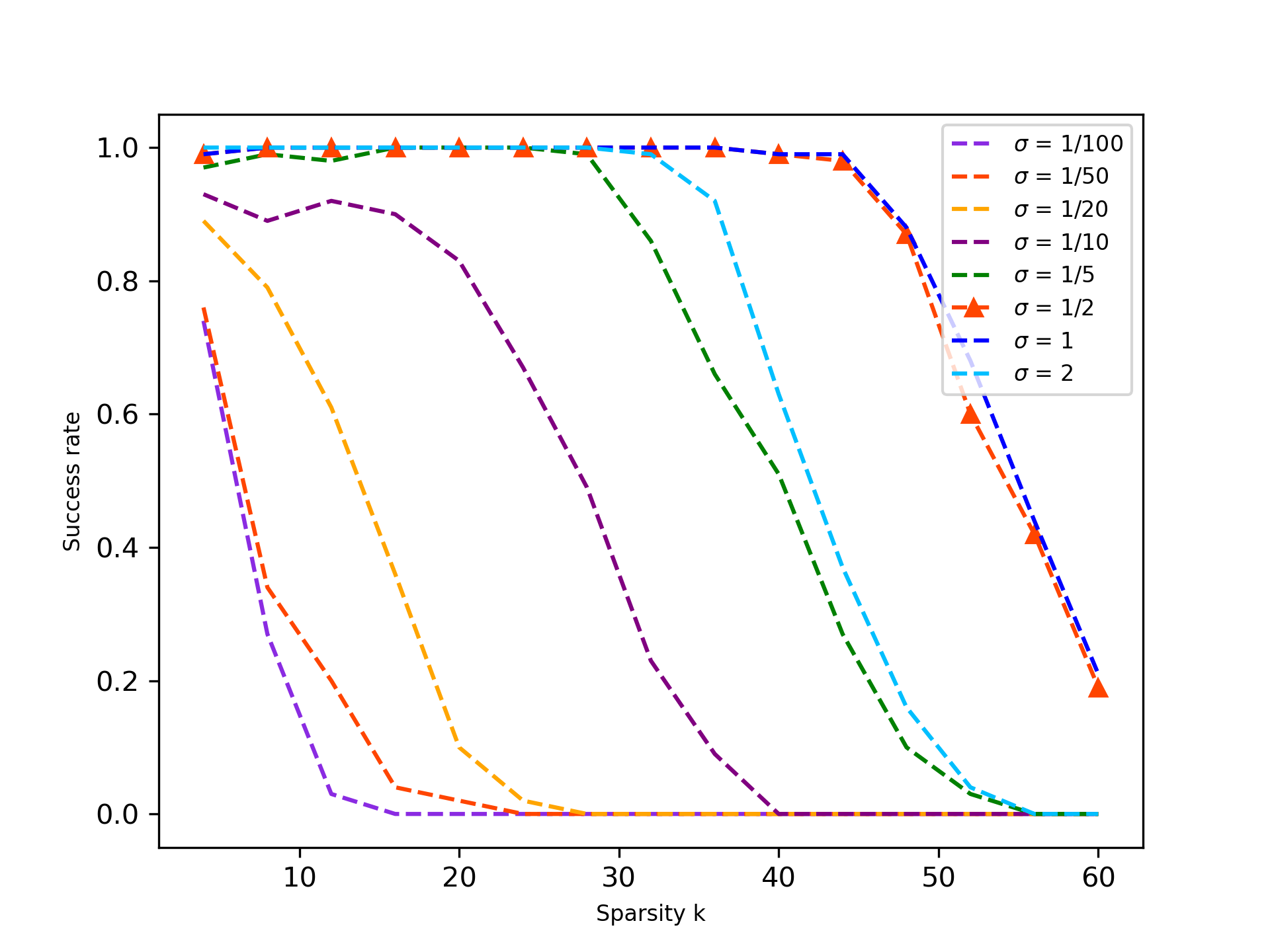}}
\hfil
\subfloat[]{\includegraphics[width=2.3in]{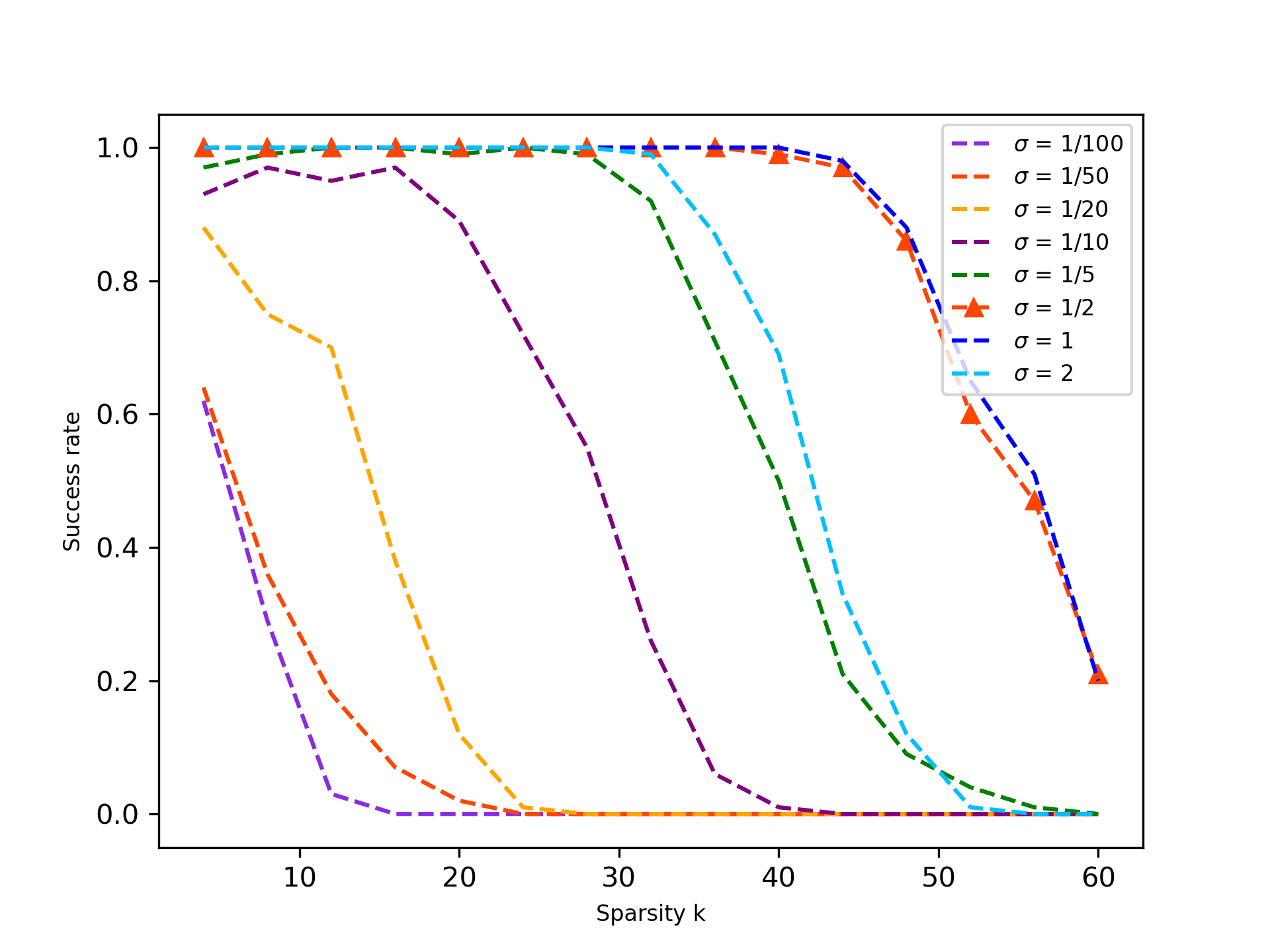}}
\hfil
\subfloat[]{\includegraphics[width=2.3in]{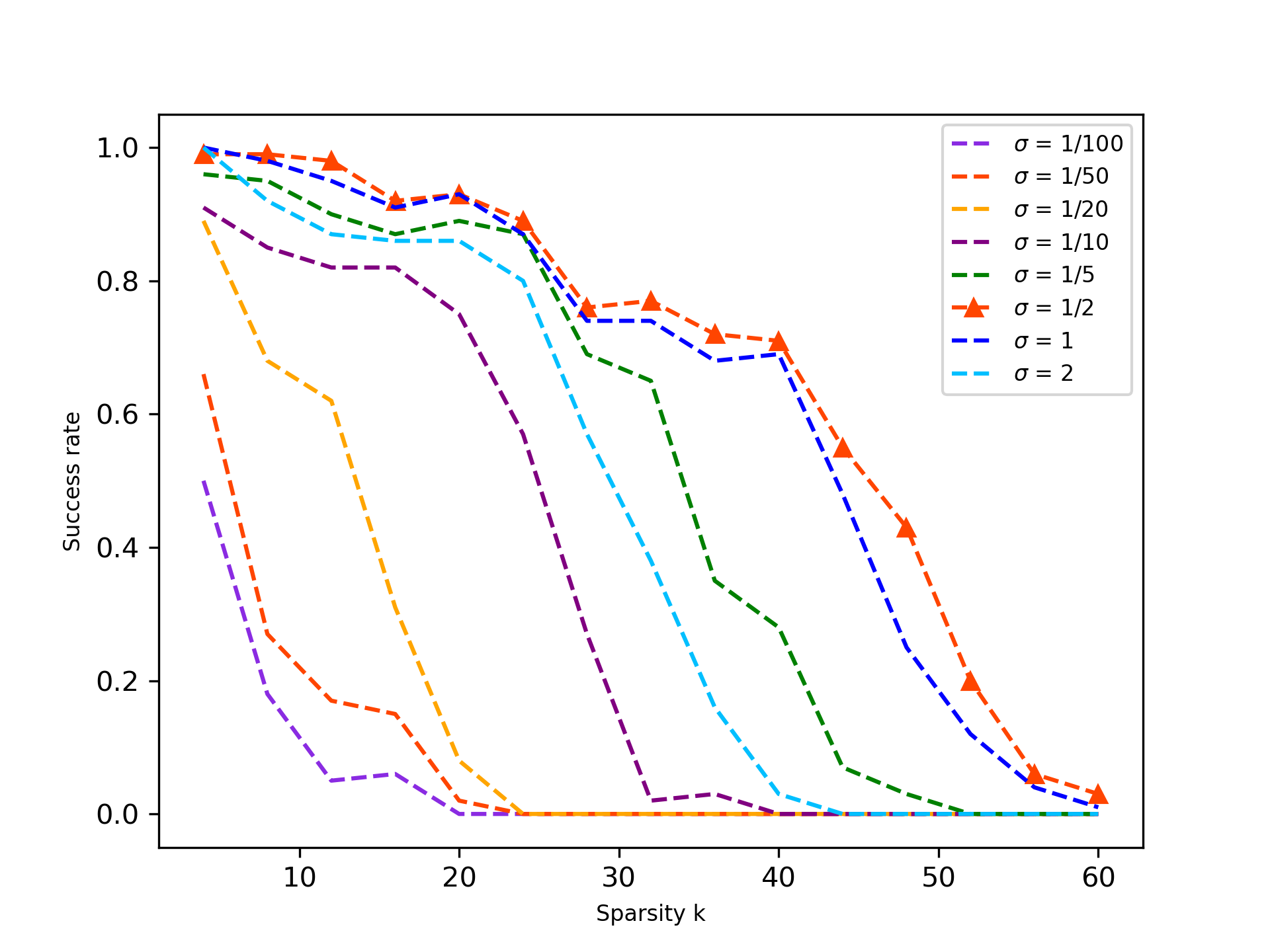}}
\hfil
\caption{Success rate of ISTA using different $\sigma$ with $\mu = 0.99\mu_{max}$ and $\lambda=1/100$ via PiE. (a) Gaussian matrix; (b) DCT matrix with $F=3$; (c) DCT matrix with $F=10$.}
\label{Fig6}
\end{figure*}

\begin{figure*}[htbp]
\centering
\subfloat[]{\includegraphics[width=2.3in]{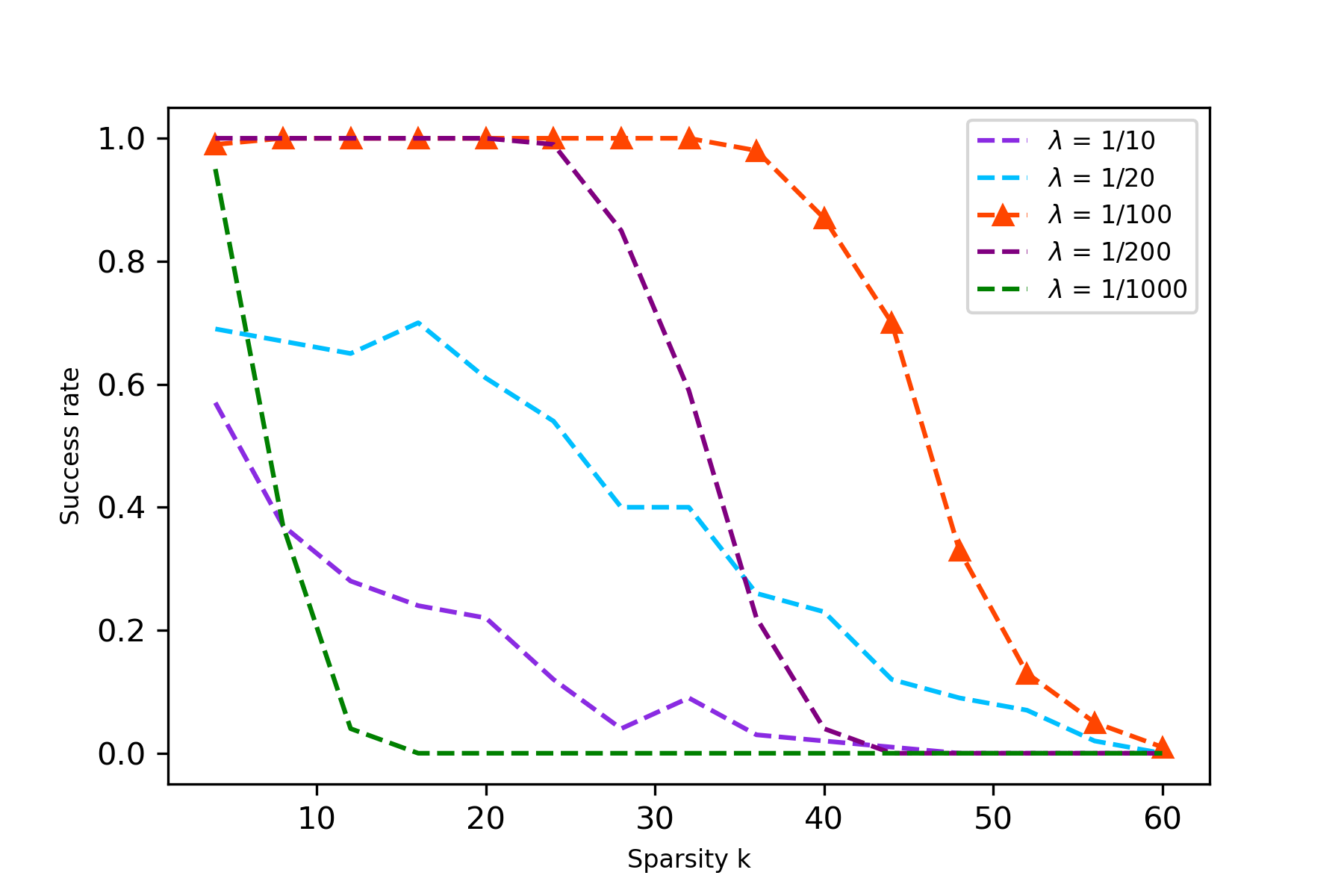}}
\hfil
\subfloat[]{\includegraphics[width=2.3in]{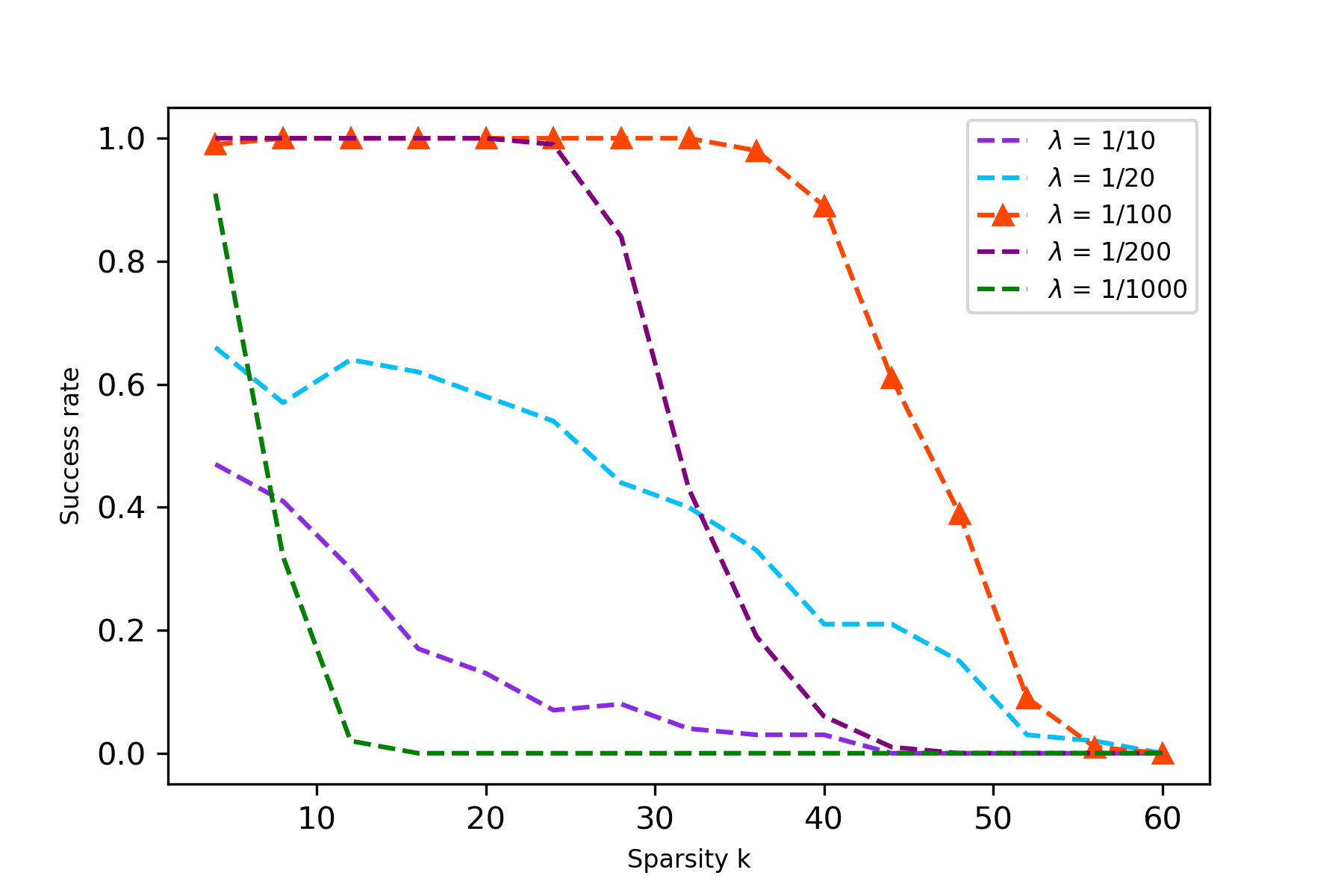}}
\hfil
\subfloat[]{\includegraphics[width=2.3in]{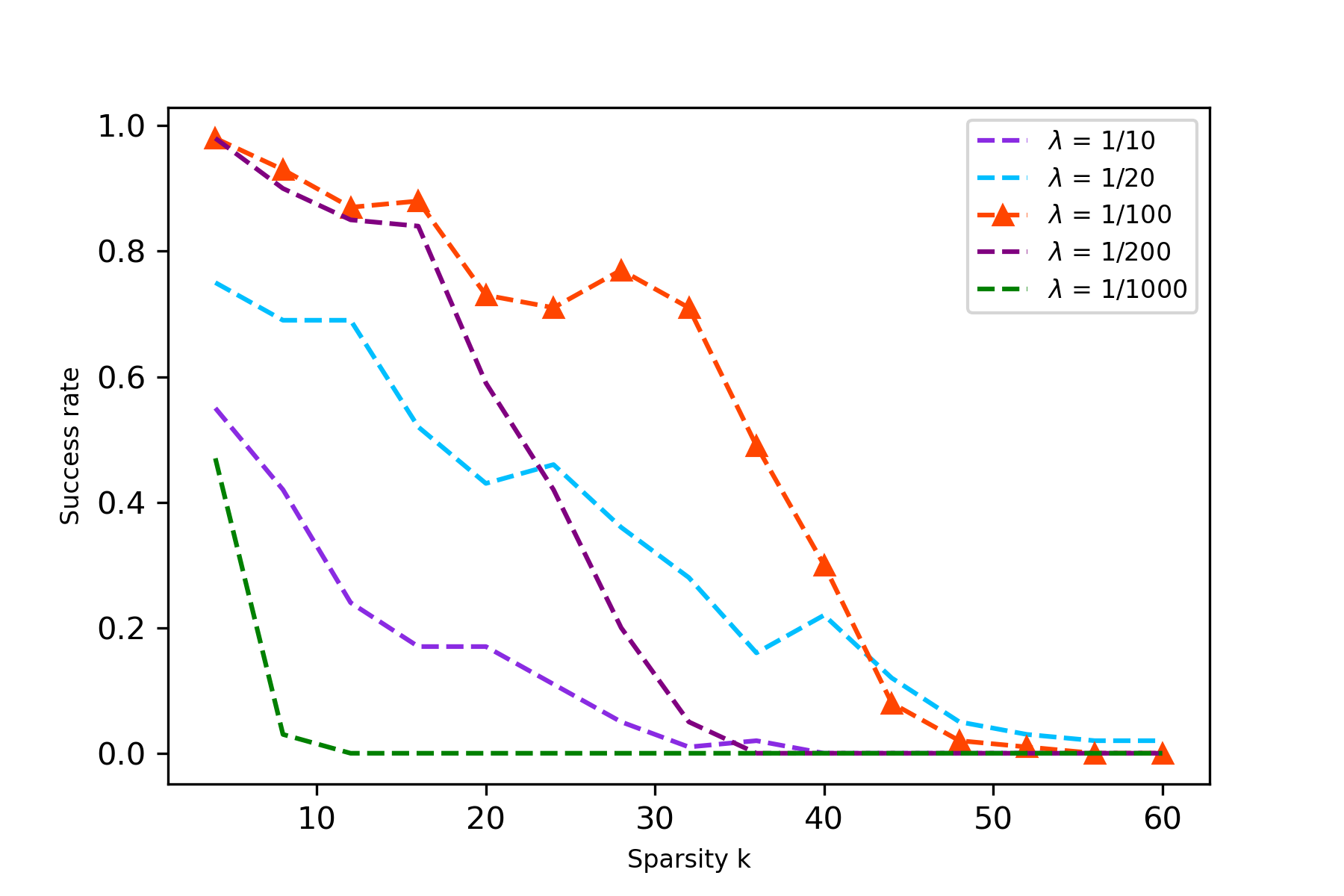}}
\hfil
\caption{Success rate of ISTA using different  $ \lambda$  with $\sigma=1/2$ and $\mu = 0.5\mu_{max}$ via PiE. (a) Gaussian matrix; (b) DCT matrix with $F=3$; (c) DCT matrix with $F=10$.}
\label{Fig7}
\end{figure*}

\begin{figure*}[htbp]
\centering
\subfloat[]{\includegraphics[width=2.3in]{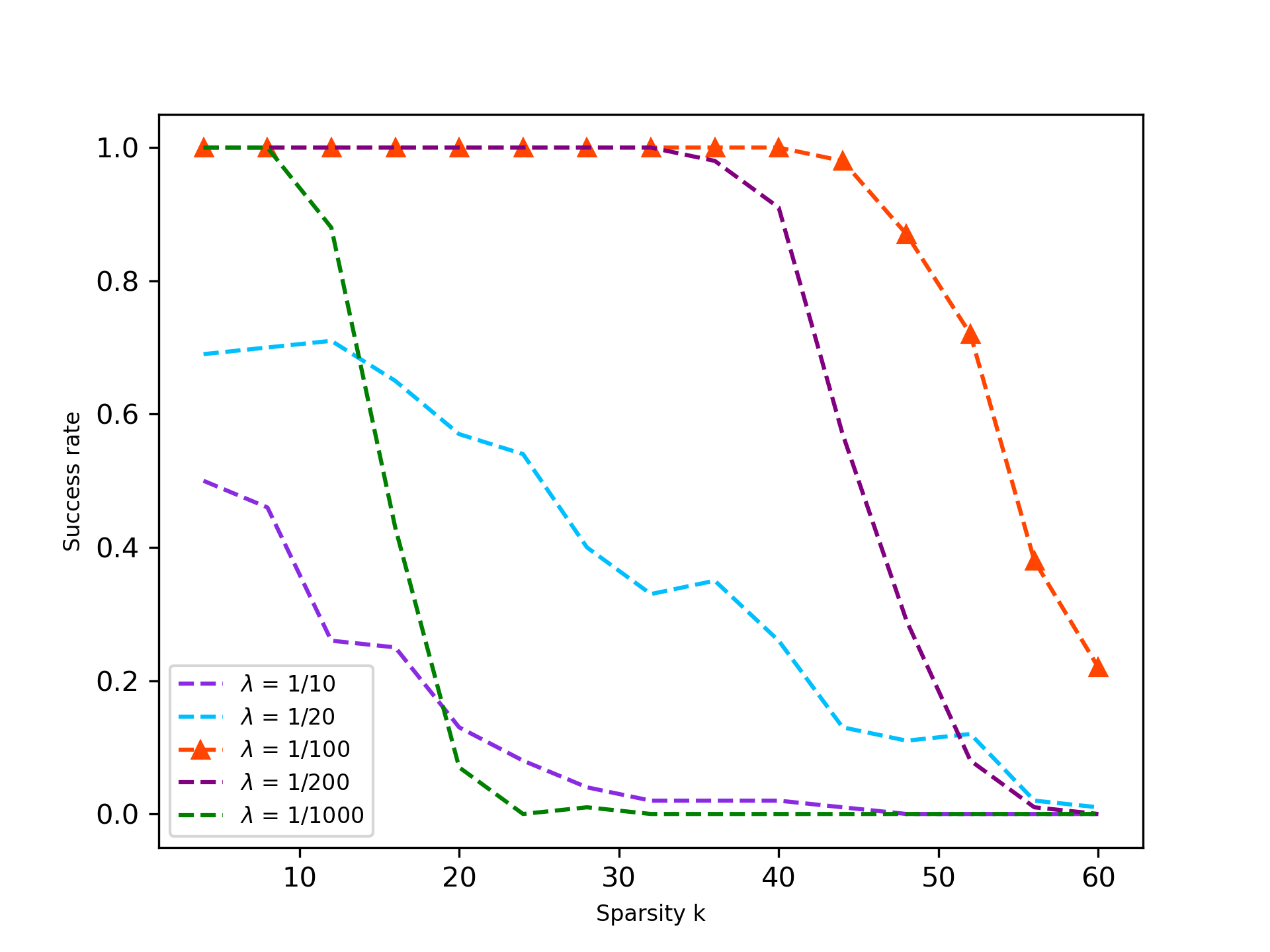}}
\hfil
\subfloat[]{\includegraphics[width=2.3in]{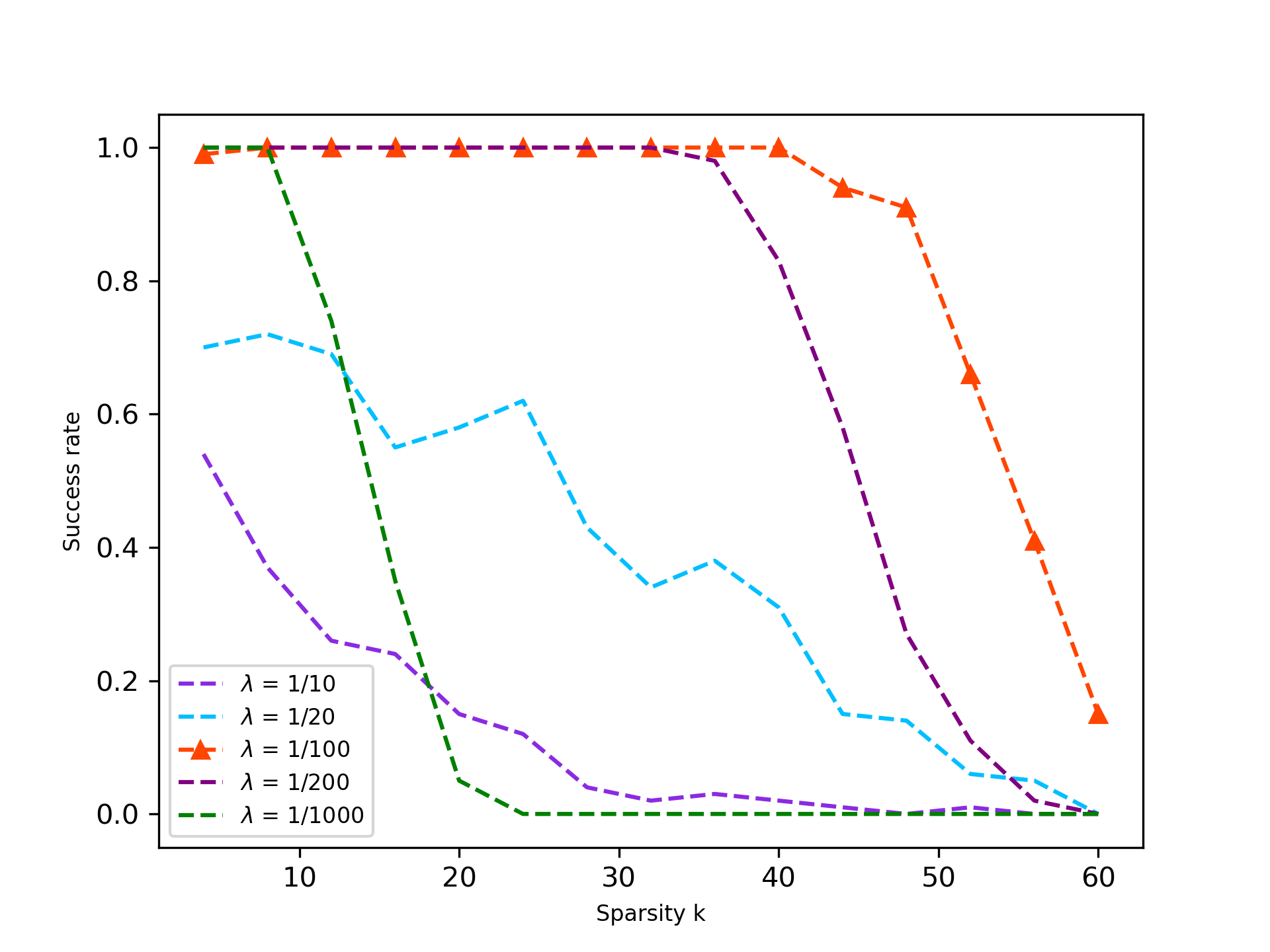}}
\hfil
\subfloat[]{\includegraphics[width=2.3in]{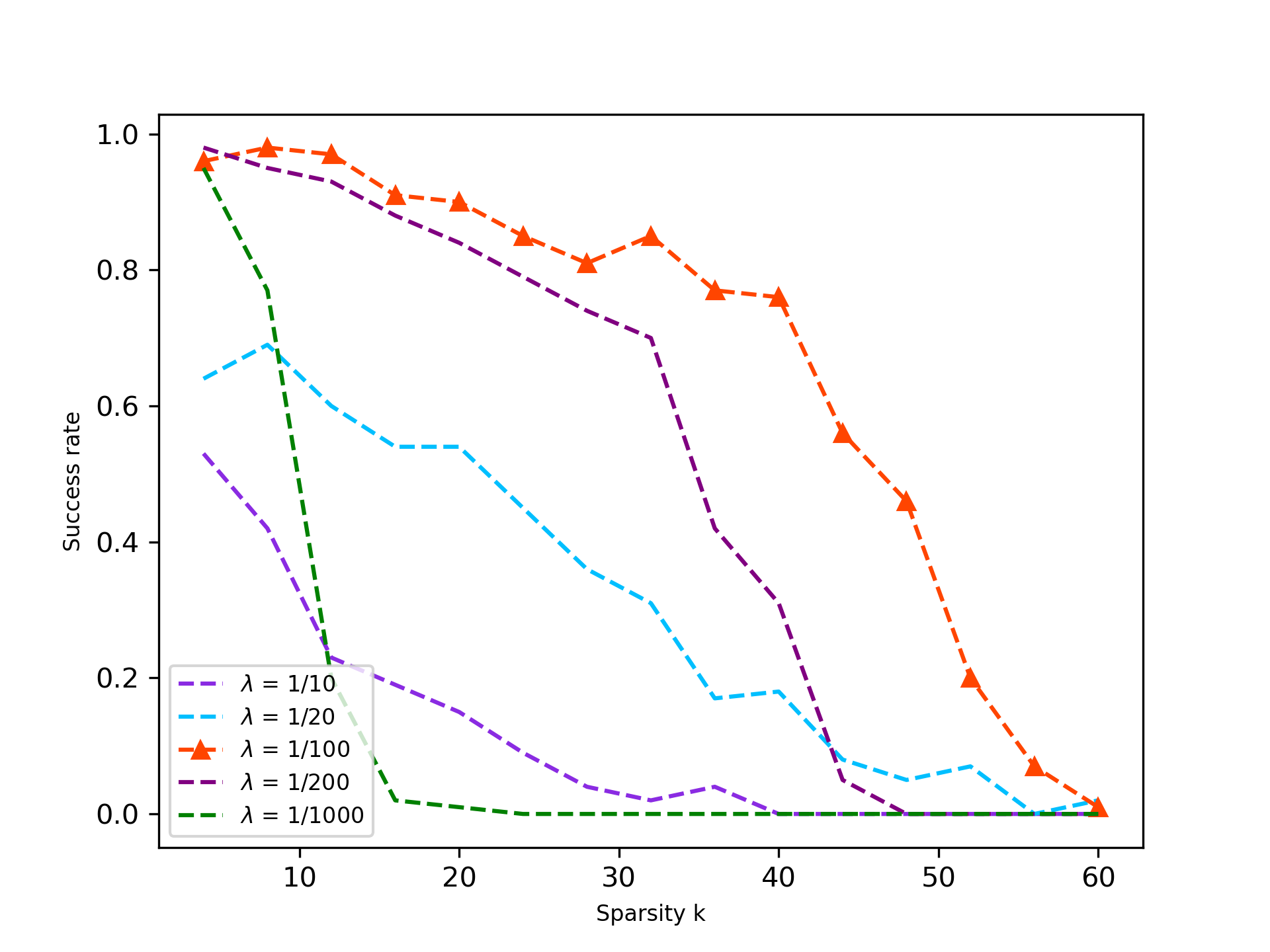}}
\hfil
\caption{Success rate of ISTA using different  $ \lambda$  with $\sigma=1/2$ and $\mu = 0.99\mu_{max}$  via PiE. (a) Gaussian matrix; (b) DCT matrix with $F=3$; (c) DCT matrix with $F=10$.}
\label{Fig8}
\end{figure*}

\begin{figure*}[htbp]
\centering
\subfloat[]{\includegraphics[width=2.3in]{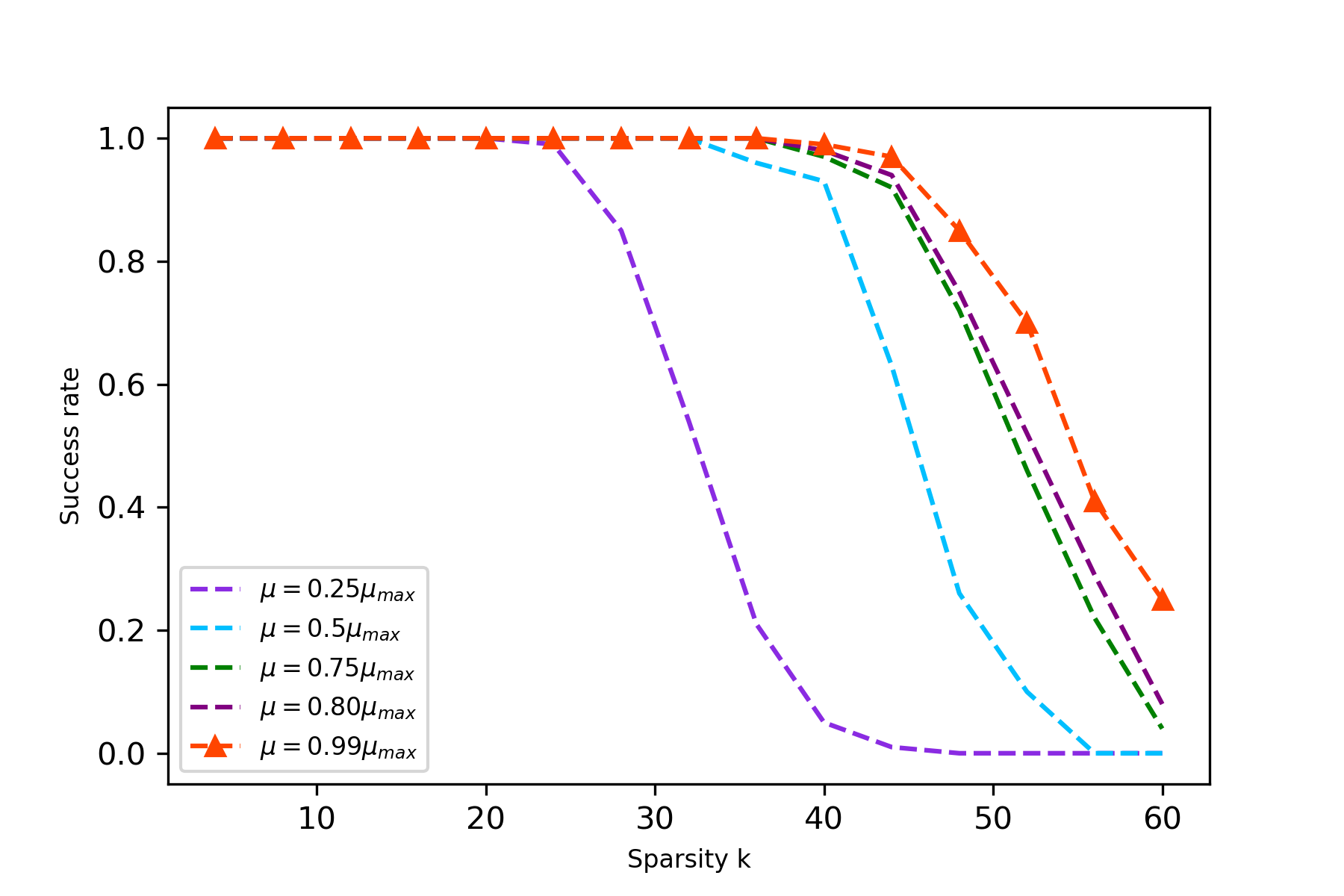}}
\hfil
\subfloat[]{\includegraphics[width=2.3in]{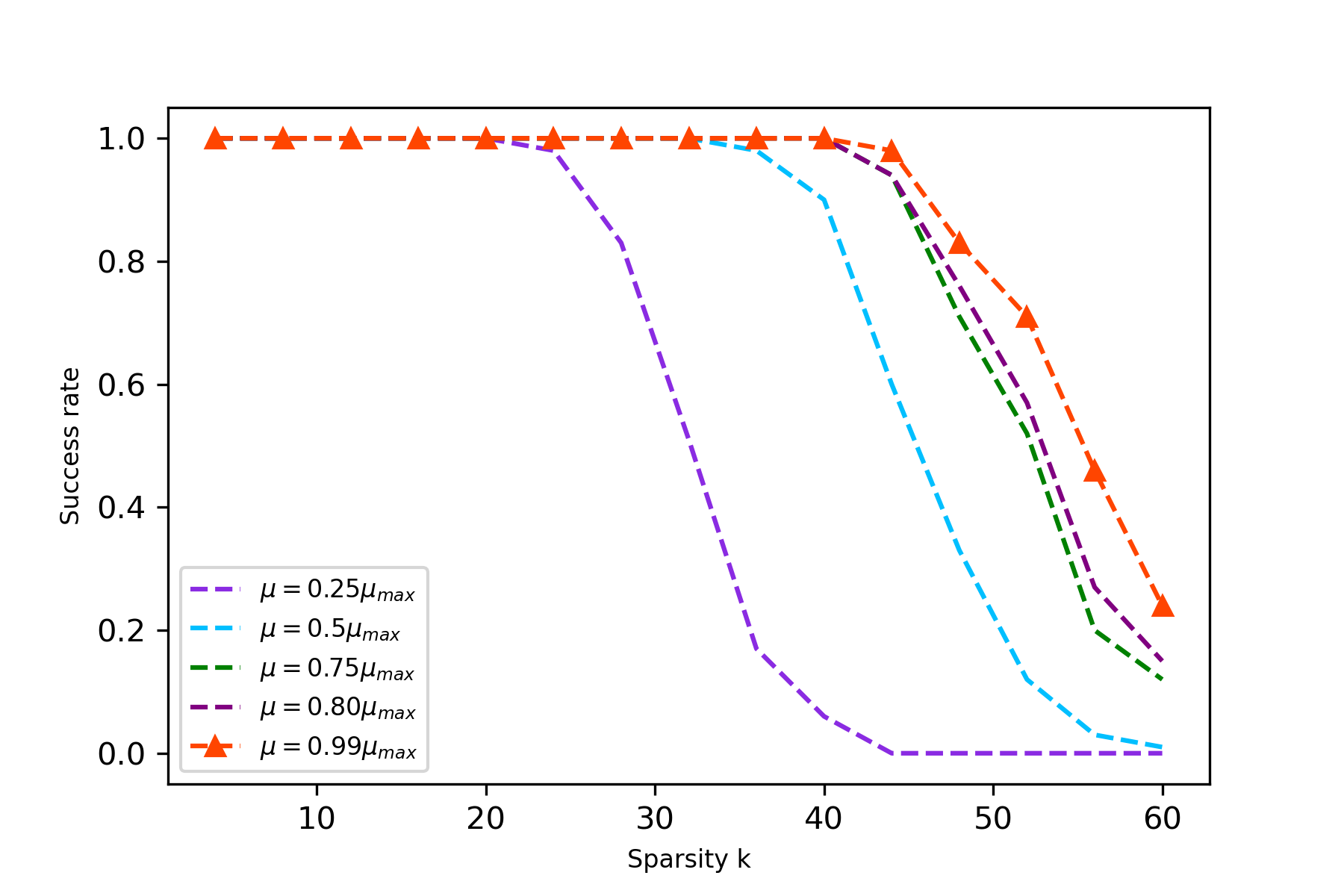}}
\hfil
\subfloat[]{\includegraphics[width=2.3in]{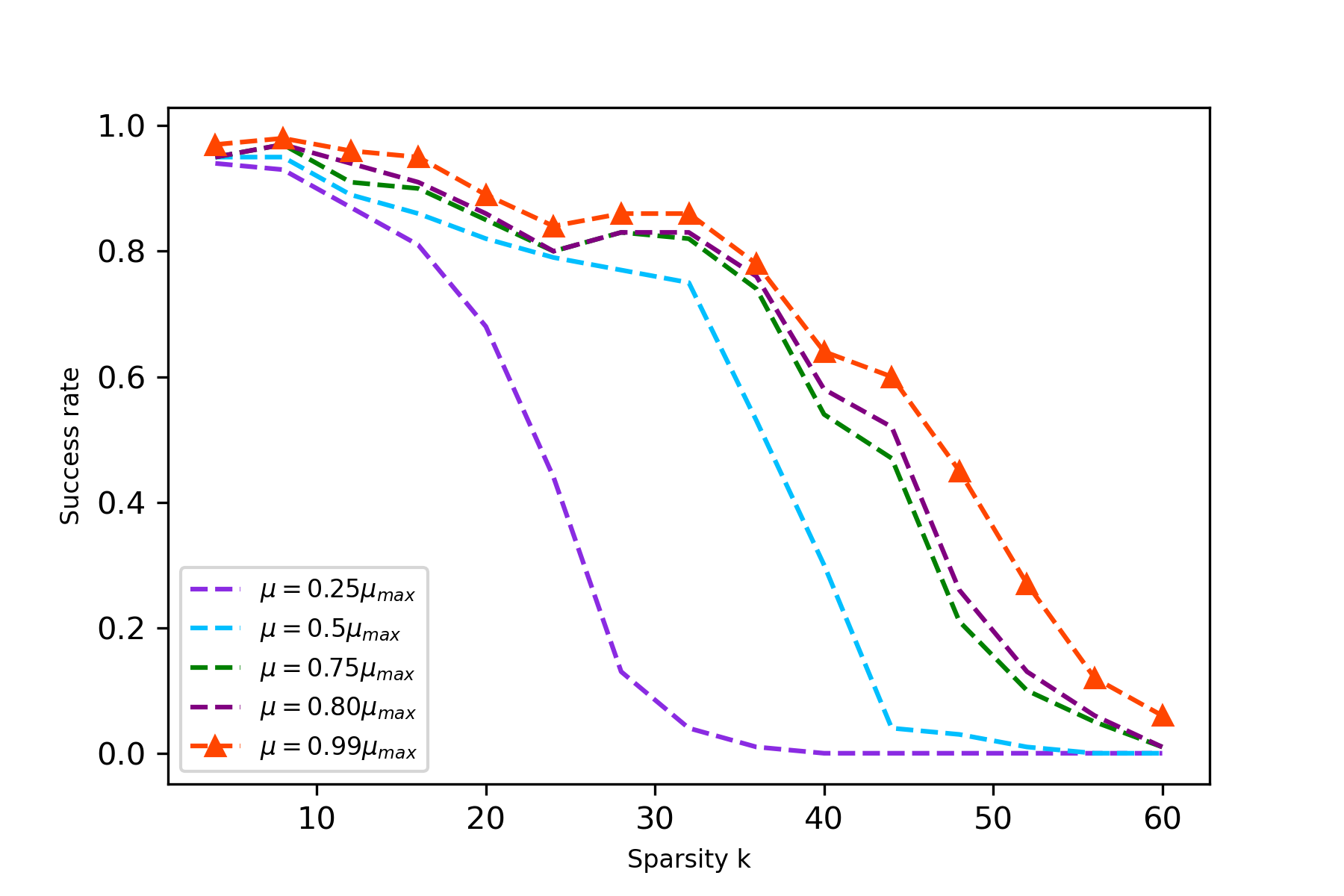}}
\hfil
\caption{Success rate of ISTA using different  $\mu\in\{0.25\mu_{max},0.5\mu_{max},0.75\mu_{max},0.8\mu_{max},0.99\mu_{max}\}$  with $\sigma = 1/2 ,  \lambda = 1/100 $ via PiE. (a) Gaussian matrix; (b) DCT matrix with $F=3$; (c) DCT matrix with $F=10$.}
\label{Fig9}
\end{figure*}

\subsection{ Comparative analysis of various penalty functions}\label{Subsection:compare}

We will compare the reconstruction success rate and time
efficiency of ISTA with PiE with other popular separable $\ell_0$-norm surrogate penalties. 
All the experiments below are run on a PC with Intel i5 CPU 2.5 GHz and 16 GB RAM, the numerical results are obtained via Jupyter Notebook 6.4.12.

As the proximal operator of the $\ell_0$-norm, the $\ell_{1/2}$-norm, and the $\ell_1$-norm correspond to the hard thresholding function, the half thresholding function, and the soft thresholding function, respectively, we follow the notations in \cite{Xu2012,Xu2023} to denote the $\ell_0$-norm, the $\ell_{1/2}$-norm, and the $\ell_1$-norm by hard, half, and soft, respectively. For comparison, penalties will be tested including PiE, soft, hard, half, SCAD, MCP, Log, TL1, and CaP.
For any given $x_0\in \bR$, 
the expressions of them and corresponding proximal operators can be found in TABLE \ref{Table2}. The regularization parameter $\lambda$, the shape parameter ($\sigma$ or $a$) of various penalties (if applicable), and the maximum step size $\mu_{\max}$ are specified for ISTA in TABLE \ref{Table3}, where $\nu_{\max}$ stands for $\nu_{\max}(A^{\top}A)$ for simplicity. The value $\nu_{\max}$ for random Gaussian matrix, DCT matrix with $F=3$, and DCT matrix with $F=10$, is approximately $5.62$ with the standard variance $0.13$, $5.68$ with the standard variance $0.15$, and $7.70$ with the standard variance $0.24$, respectively.

Let us explain the choices of those parameters in TABLE \ref{Table3}.
The shape parameter $a=3.7$ was suggested for SCAD and MCP in \cite{Fan2001,Xu2023} and $a\in\{0.1,0.5\}$ was used for Log in  \cite{Prater2022}.
Remember that for a $\rho$-weakly convex penalty, the maximum step size of ISTA is $\mu_{\max}=\frac{2}{\nu_{\max} +\rho}$. The minimum values $\rho$ for $\rho$-weakly convex penalties PiE, soft, SCAD, MCP, Log, TL1 in TABLE \ref{Table3} to be convex, are $\frac{\lambda}{\sigma^2}$, $0$, $\frac{1}{a-1}$, $\frac{1}{a}$, $\frac{\lambda}{a^2}$, and $\frac{2(a+1)\lambda}{a^2}$, respectively. 
Compared to $\nu_{\max}$, the aforementioned $\rho$ is generally a relatively small constant, which implies that $\mu_{\max}$ is just slightly less than $\frac{2}{\nu_{\max}}$. Notice that penalties including hard, half, and CaP are not weakly convex. For those penalties, the maximum step size in our experiments is set to be $\frac{2}{\nu_{\max}}$. Notice that $\mu$ in the ISTA is normally taken as in the interval $(0,\frac{1}{\nu_{\max}})$ in the literature, see, for instance, in \cite[(48)]{Xu2012} for half, in \cite[Algorithm 2]{Chen2016} for the $\ell_p$-norm, in \cite[Algorithm 1]{Prater2022} for Log, in \cite[Theorem 4.1]{Zhang2017} for TL1, and in \cite[Theorem 3]{Lu2015} for a general penalties.

The success rate and the convergence time of ISTA with nine penalties for three different types of sensing matrices are illustrated in Fig. \ref{Fig10} and Fig. \ref{Fig11} for the two step sizes $\mu=0.5\mu_{\max}$ and  $\mu=0.99\mu_{\max}$, respectively. All $\mu_{\max}$ of various penalties were listed in TABLE \ref{Table3}. As shown in Fig. \ref{Fig10} with $\mu=0.5\mu_{\max}$, Log (deep sky blue dashed line) and PiE (orangered dashed line with triangle down marker) rank first and second among the nine penalties. 
Using a large step size of $\mu=0.99\mu_{\max}$, Fig. \ref{Fig11} shows PiE and Log rank first and second among the nine penalties. As illustrated in the left columns of Fig. \ref{Fig10} and Fig. \ref{Fig11}, the large step size $\mu=0.99\mu_{\max}$ will significantly increase the success rate of ISTA with $\rho$-weakly convex penalties PiE and Log  whenever the sparsity levels are in the interval $[40,60]$.
Additionally, the large step size slightly improves the success rate of ISTA with SCAD (purple dashed line), MCP (green dashed line), and half (orange dashed line).
The convergence time of various penalties is very close, as depicted in the right columns of Fig. \ref{Fig10} and Fig. \ref{Fig11}.

\begin{table}[!hbt]
\centering
\setlength{\tabcolsep}{0pt}
\caption{List of some $\ell_0$-norm surrogate penalties and their proximal operators}
\label{Table2}
\resizebox{\linewidth}{!}{
\begin{tabular}{|c|c|c|} \hline
Penalties & Formulas  & Proximal operators\\ \hline
PiE & $\lambda (1-e^{-\frac{|x|}{\sigma}})$ & Theorems \ref{ProximalCase1} and \ref{Coratau} \\ \hline
soft &  $\lambda |x|$    & $\left\{\begin{array}{ll} \{0\}, & \mbox{ if }|x_0|\le \lambda\\
     \{x_0-\sign(x_0)\lambda\}, & \mbox{ if }|x_0|>\lambda\\
 \end{array}
 \right.$ \\ \hline
hard &  $\lambda |x|_0$ & $\left\{\begin{array}{ll}
    \{0\}, & \mbox{ if }|x_0|<\sqrt{2\lambda}\\
    \{0,x_0\}, & \mbox{ if }|x_0|=\sqrt{2\lambda}\\
    \{x_0\}, & \mbox{ if }|x_0|>\sqrt{2\lambda}\\
\end{array}
\right.
$\\ \hline
half &  $\lambda |x|^{\frac12}$ & $\left\{\begin{array}{ll}
    \{0\}, & \mbox{if }  |x_0|<\frac{3}{2}\lambda^{\frac{2}{3}}\\
    \{0,\sign(x_0) \lambda^{\frac{2}{3}}\}, & \mbox{if }  |x_0|=\frac{3}{2}\lambda^{\frac{2}{3}} \\
   \big\{\frac{2}{3}x_0\big(1\!+\!\cos(\frac{2}{3}\arccos(\!-\frac{3^{\frac{3}{2}}}{4}\lambda|x_0|^{-\frac{3}{2}}))\big)\big\},  & \mbox{if }  |x_0|>\frac{3}{2}\lambda^{\frac{2}{3}}\\
\end{array}
\right.
$\\ \hline
SCAD & $\left\{\begin{array}{ll}
  \lambda|x|,   & \mbox{ if }|x|\le \lambda  \\
  \frac{-x^2+2a\lambda |x|-\lambda^2}{2(a-1)}, \!  & \mbox{ if }\lambda<|x| \le a\lambda\\
  \frac{(a+1)\lambda^2}{2},  & \mbox{ if }|x|>a\lambda
\end{array}
\right.$   & $\left\{\begin{array}{ll}
    \{\sign(x_0)(|x_0|-\lambda)_+\}, & \mbox{ if }|x_0|\le 2\lambda \\
    \{\frac{(a-1)x_0-\sign(x_0)a\lambda}{a-2}\}, &\mbox{ if }2\lambda<|x_0|\le a\lambda\\
    \{x_0\},&\mbox{ if }|x_0|\ge a\lambda
\end{array}
\right.$ ($a>2$)\\ \hline
MCP & $\left\{\begin{array}{ll}
\lambda |x|-\frac{x^2}{2a}, & \mbox{ if }|x|\le a\lambda\\
\frac{a\lambda^2}{2}, & \mbox{ if }|x|> a\lambda
\end{array}
\right.$    & $\left\{\begin{array}{ll}
\{\frac{\sign(x_0)(|x_0|-\lambda)_+}{1-1/a}\}, &\mbox{ if }|x_0|\le a\lambda\\
\{x_0\}, & \mbox{ if }|x_0|> a\lambda
\end{array}
\right.$ ($a>1$)\\ \hline
Log & $\lambda\log\big(1+ \frac{|x|}{a}\big)$   & $\left\{\begin{array}{ll}
\{0\}, & \mbox{ if } |x_0|\le \frac{\lambda}{a}\\
\{\sign(x_0)r(|x_0|)\}, & \mbox{ if } |x_0|> \frac{\lambda}{a}
\end{array}
\right.$ ($\sqrt{\lambda}\le a$) \\ 
 && $\left\{\begin{array}{ll}
\{0\}, & \mbox{ if }|x_0|< \bar{x}_0\\
\{0,\sign(\bar{x}_0)r(\bar{x}_0)\}, &
 \mbox{ if } |x_0|=\bar{x}_0\\
\{\sign(x_0)r(x_0)\}, & \mbox{ if } |x_0|>\bar{x}_0
\end{array}
\right.$ ($\sqrt{\lambda}>a$)\\ 
&& where $r(x_0):=\frac{x_0-a}{2}+\sqrt{\frac{(x_0+a)^2}{4}-\lambda}$ for any $x_0\ge\max\{2\sqrt{\lambda}-a,0\}$ \\
& &$\bar{x}_0\in[2\sqrt{\lambda}-a,\frac{\lambda}{a}]$ is the only root of $\frac{1}{2}r^2(x_0)-x_0 r(x_0)+\lambda\log(1+\frac{|r(x_0)|}{a})=0$\\ \hline

TL1 &  $\lambda\frac{(a+1)|x|}{a+|x|}$  & $\left\{\begin{array}{ll}
\{0\},& |x_0|\le\frac{\lambda(a+1)}{a}\\
\{\sign(x_0)g_{\lambda}(|x_0|)\}, & \mbox{ otherwise}
\end{array}
\right.$ ($\lambda\le \frac{a^2}{2(a+1)}$)\\ 
&  & $\left\{\begin{array}{ll}
\{0\},& |x_0|< \sqrt{2\lambda(a+1)}-\frac{a}{2}\\
\sign(x_0)\{0,\sqrt{2\lambda(a+1)}-a\},& |x_0|= \sqrt{2\lambda(a+1)}-\frac{a}{2}\\
\{\sign(x_0) g_{\lambda}(|x_0|)\}, & \mbox{ otherwise}
\end{array}
\right.$ ($\lambda>\frac{a^2}{2(a+1)}$) \\
&& where $g_{\lambda}(x_0)=\frac{2}{3}(a+|x_0|)\cos\big(\frac13 \arccos\big(1-\frac{27\lambda a(a+1)}{2(a+|x_0|)^3}\big)\big)-\frac{2a}{3}+\frac{|x_0|}{3}$\\ \hline
CaP &  $\lambda\min\{|x|,a\}$  & $\left\{\begin{array}{ll}
\{0\}, &  \mbox{ if } |x_0|<\lambda\\
\{x_0-\sign(x_0)\lambda\}, &\mbox{ if } \lambda \le |x_0|<a+\frac{\lambda}{2}\\
\{x_0,x_0-\sign(x_0)\lambda\}, & \mbox{ if } |x_0|=a+\frac{\lambda}{2}\\
\{x_0\}, & \mbox{ if } |x_0|>a+\frac{\lambda}{2}
\end{array}
\right.$  ($0<\lambda<2a$)\\  
& &$\left\{\begin{array}{ll}
\{0\}, &\mbox{ if }  |x_0|<\sqrt{2a\lambda}\\
\{0,x_0\}, &\mbox{ if } |x_0|=\sqrt{2a\lambda}\\
\{x_0\}, & \mbox{ if } |x_0|>\sqrt{2a\lambda}
\end{array}
\right.$ ($\lambda>2a$) \\  \hline
\end{tabular}}
\end{table}

\begin{table}[!hbt]
\centering
\caption{Parameters for various penalties in ISTA}
\label{Table3}
\resizebox{\linewidth}{!}{
\begin{tabular}{|c|c|c|c|c|c|c|c|c|c|} \hline
Penalties & PiE & soft & hard & half & SCAD & MCP & Log & TL1 & CaP\\ \hline
 &$\lambda=0.01$  &$\lambda=0.001$  &$\lambda=0.05$  &$\lambda=0.05$  &$\lambda=0.05$  &$\lambda=0.05$   &$\lambda=0.001$ &$\lambda=0.001$   &$\lambda=0.001$\\  
Parameters  &$\sigma=0.5$  & NA & NA & NA &$a=3.7$  &$a=3.7$  &$a=0.1$  &$a=2$   &$a=1$\\  
           &$\mu_{\max}\!=\!\frac{2}{\nu_{\max}+\frac{\lambda}{\sigma^2}}$  & $\mu_{\max}\!=\!\frac{2}{\nu_{\max}}$  & $\mu_{\max}\!=\!\frac{2}{\nu_{\max}}$ & $\mu_{\max}\!=\!\frac{2}{\nu_{\max}}$ &$\mu_{\max}\!=\!\frac{2}{\nu_{\max}+\frac{1}{a-1}}$  & $\mu_{\max}\!=\!\frac{2}{\nu_{\max}+\frac{1}{a}}$ &$\mu_{\max}\!=\!\frac{2}{\nu_{\max}+\frac{\lambda}{a^2}}$ &$\mu_{\max}\!=\!\frac{2}{\nu_{\max}+\frac{2(a+1)\lambda}{a^2}}$ &$\mu_{\max}\!=\!\frac{2}{\nu_{\max}}$\\ \hline
\end{tabular}}
\end{table}

\begin{figure*}[htbp]
\centering
\subfloat[]{\begin{minipage}[b]{0.3\textwidth}
\includegraphics[scale=0.36]{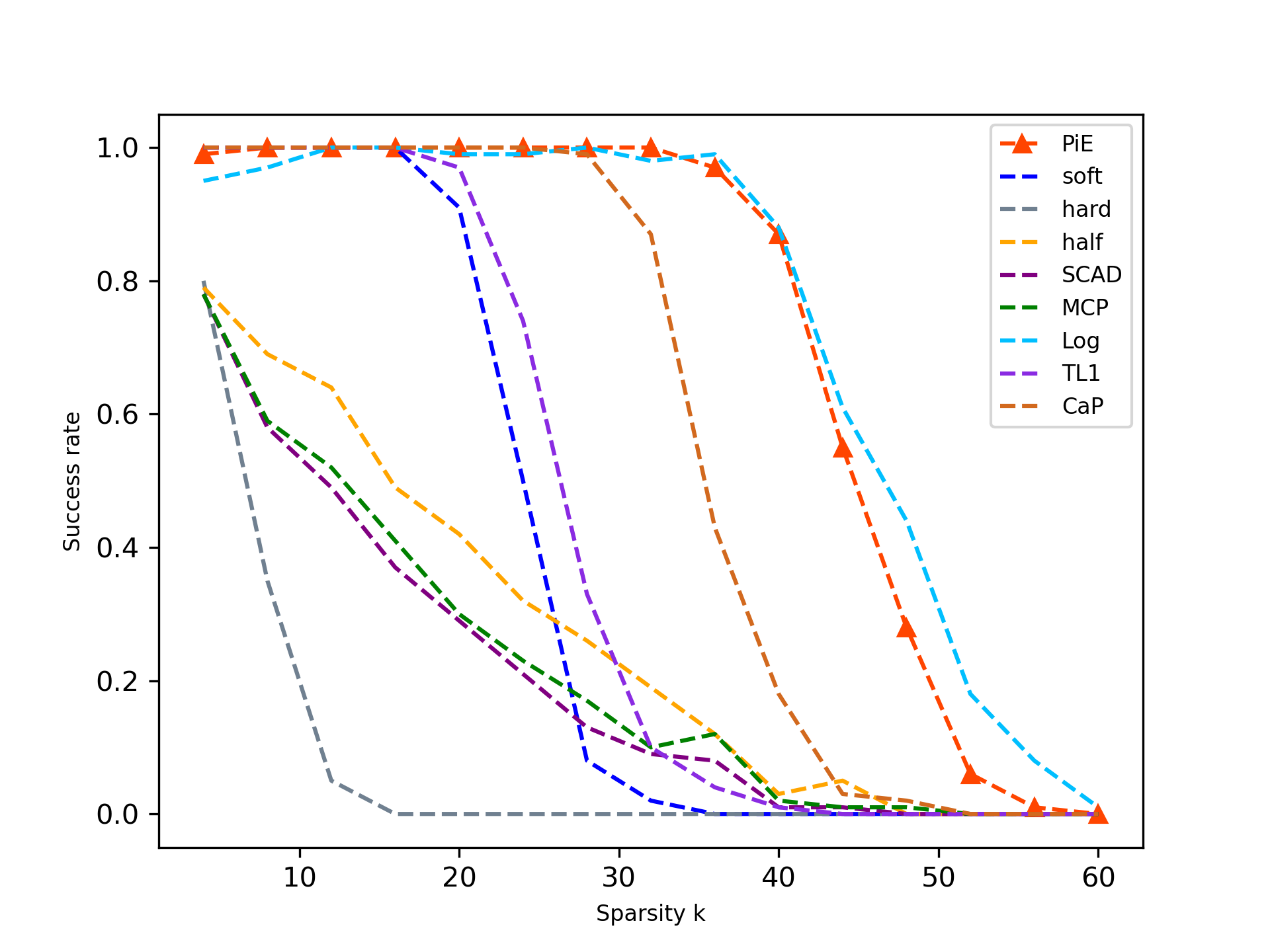}\\
         \includegraphics[scale=0.36]{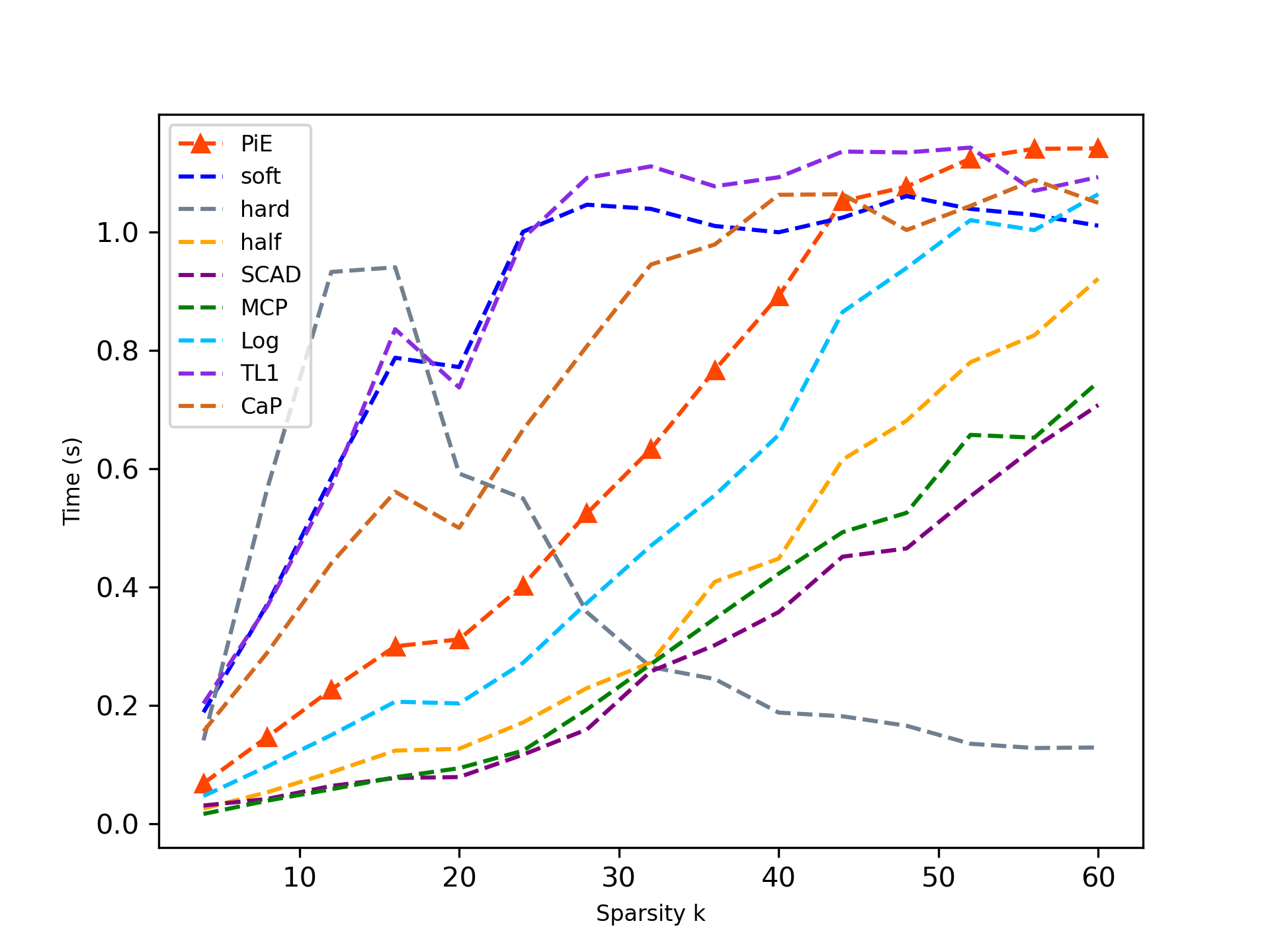}
         \end{minipage}}
\hfil
\subfloat[]{\begin{minipage}[b]{0.3\textwidth}
\includegraphics[scale=0.36]{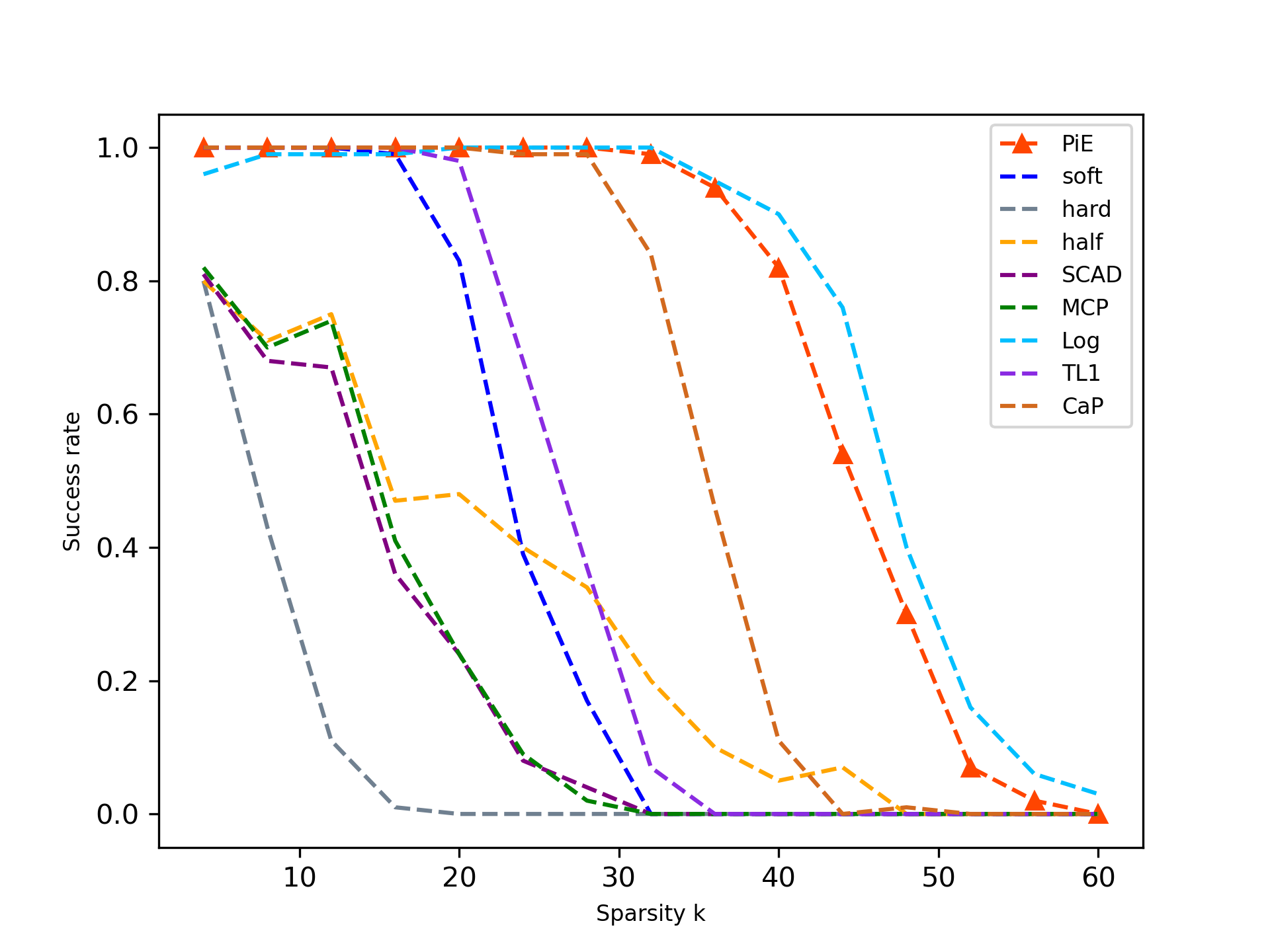}\\
         \includegraphics[scale=0.36]{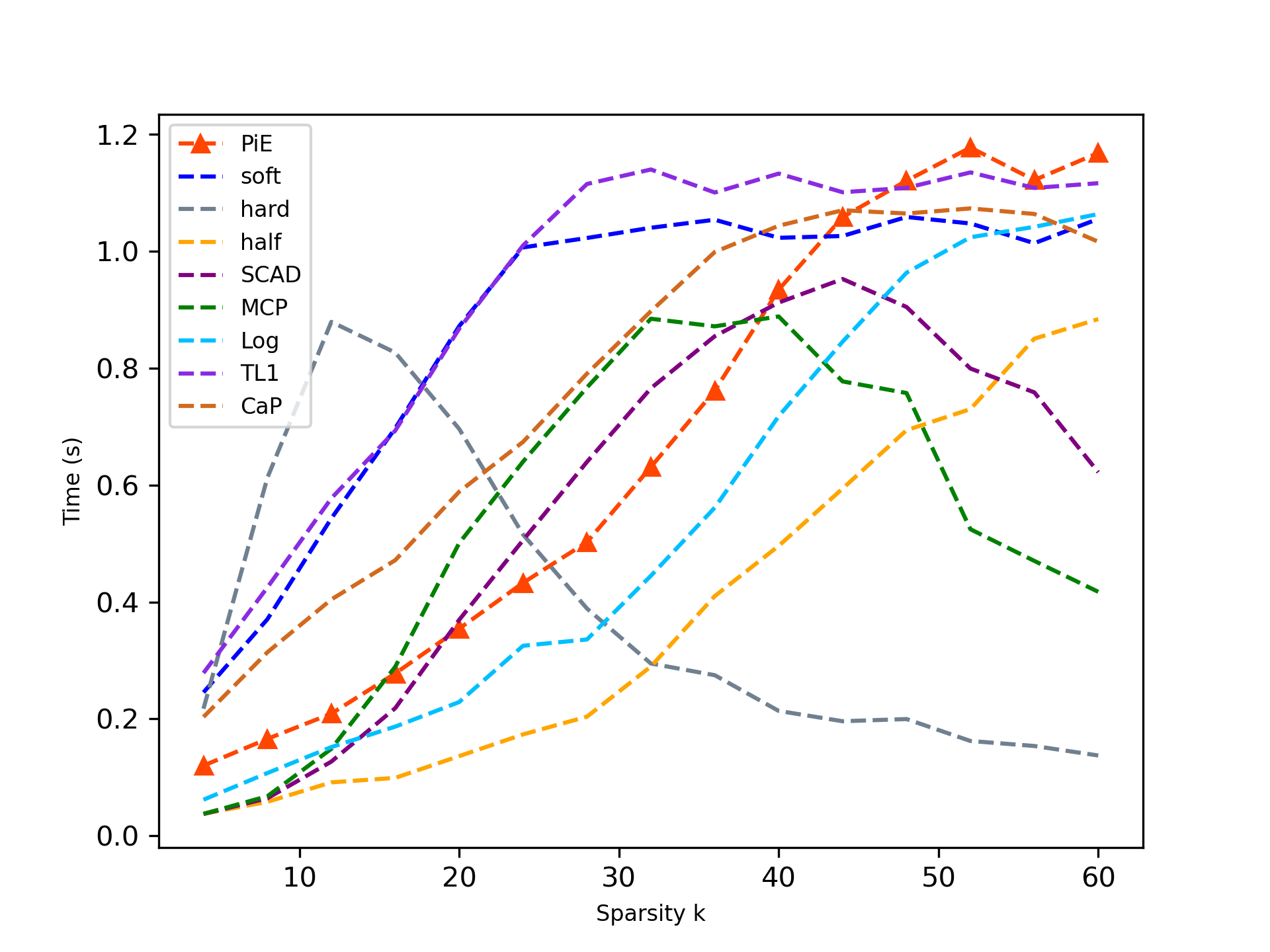}
         \end{minipage}}
\hfil
\subfloat[]{\begin{minipage}[b]{0.3\textwidth}
\includegraphics[scale=0.36]{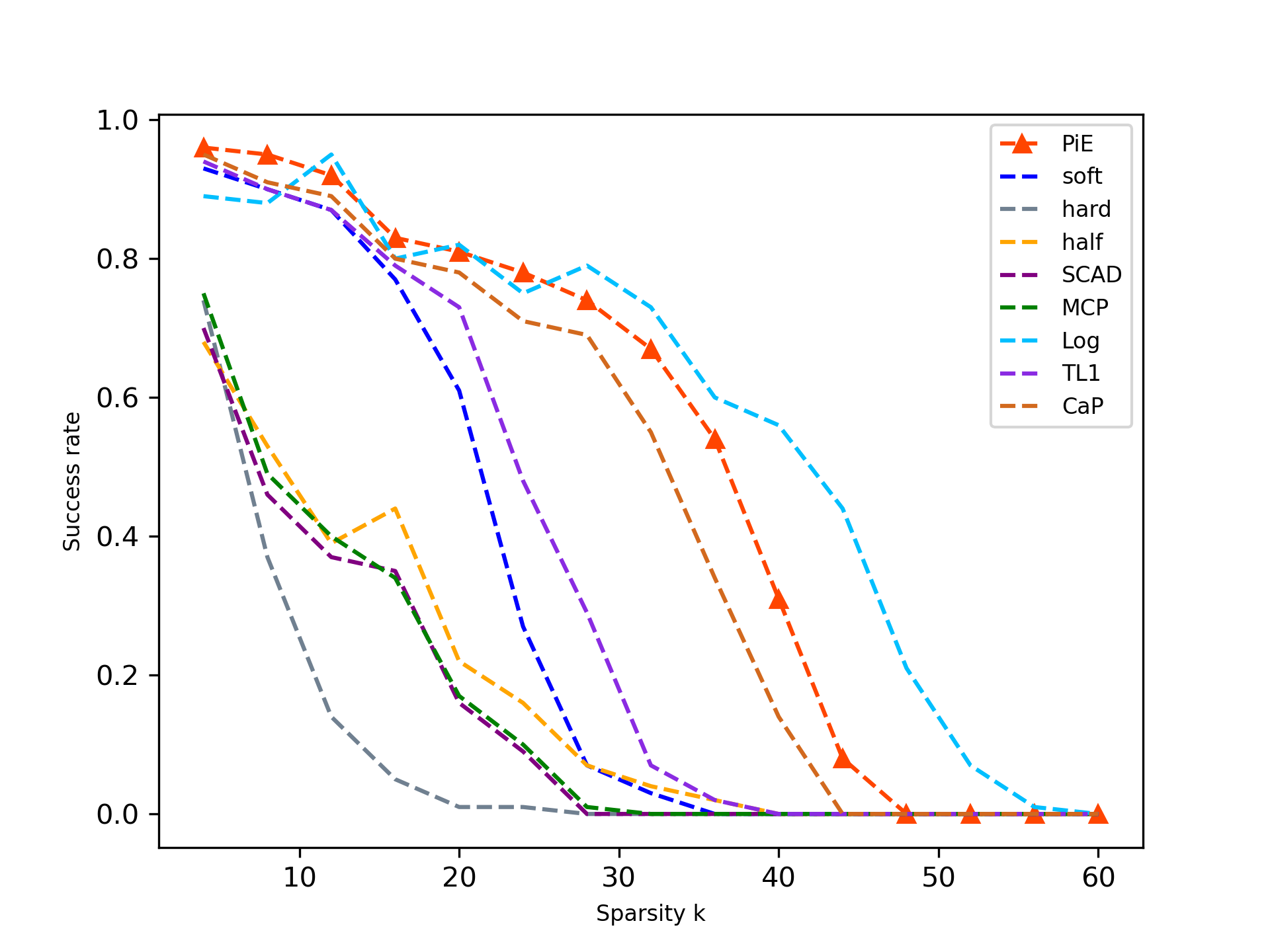}\\
         \includegraphics[scale=0.36]{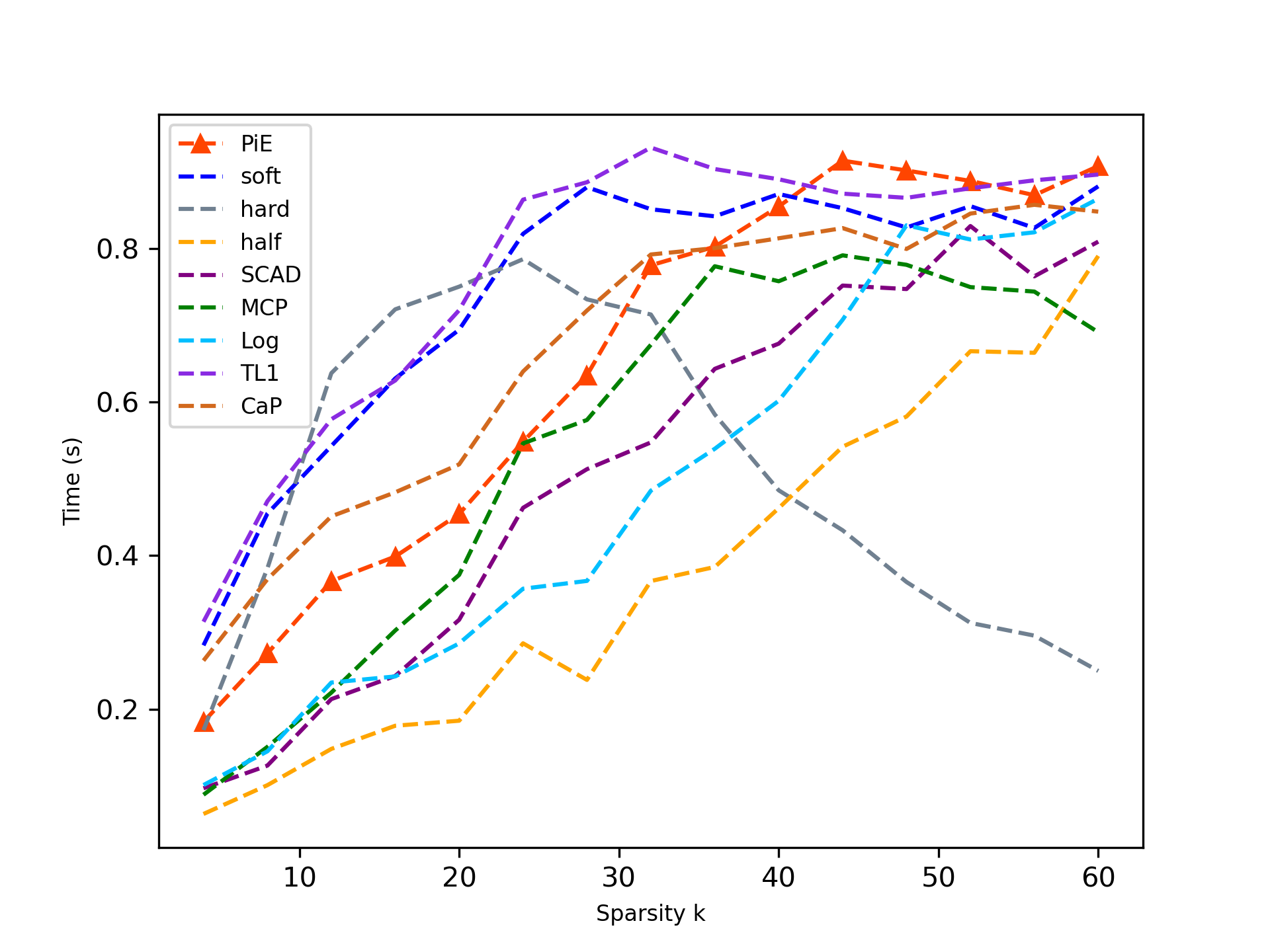}
         \end{minipage}}         
\caption{Success rate (the first row) and convergence time (the second row) of ISTA with a fixed step size $\mu = 0.5\mu_{max}$. (a) Gaussian matrix; (b) DCT matrix with $F=3$; (c) DCT matrix with $F=10$.}
\label{Fig10}
\end{figure*}

\begin{figure*}[htbp]
\centering
\subfloat[]{\begin{minipage}[b]{0.3\textwidth}
\includegraphics[scale=0.38]{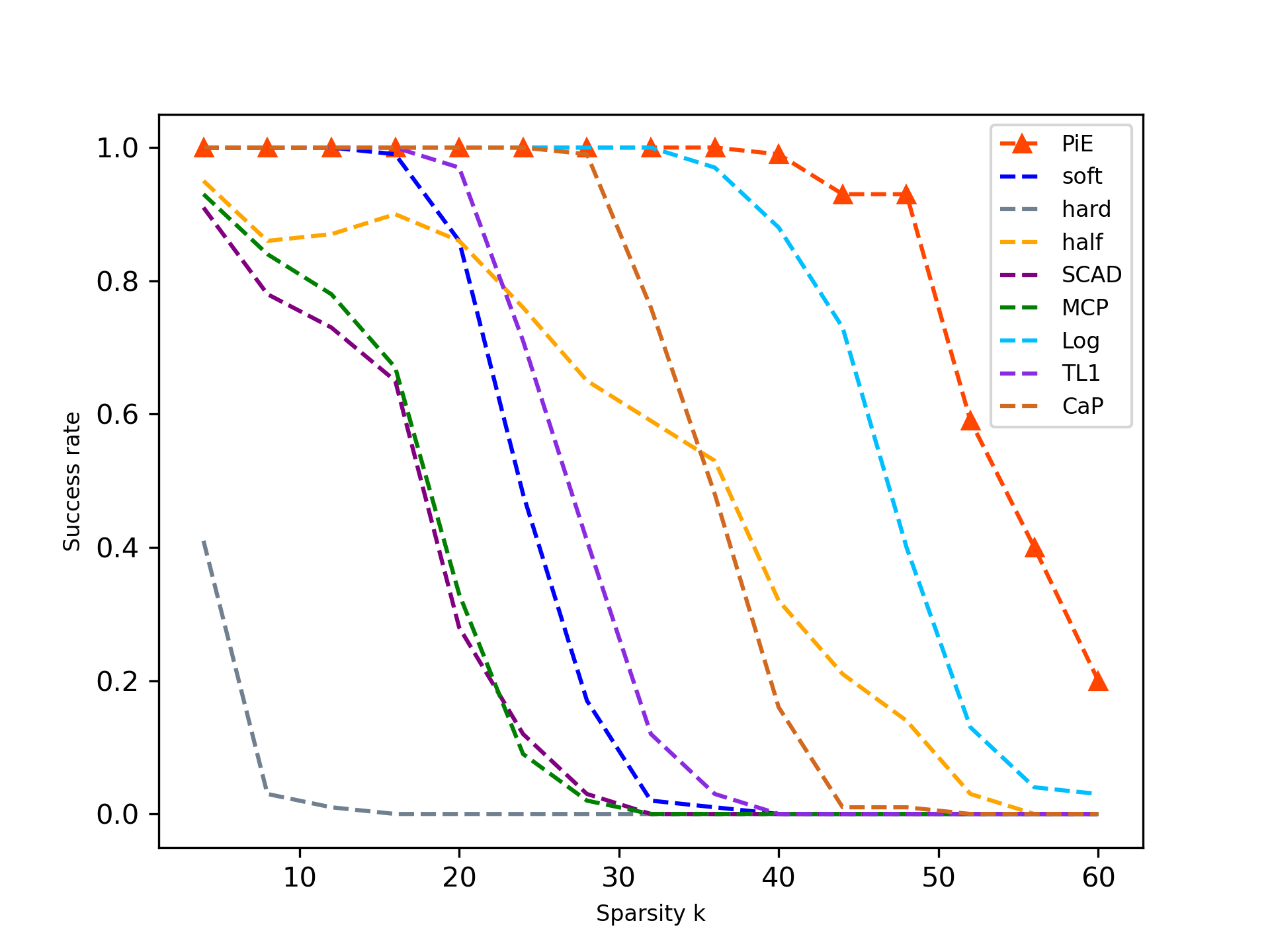}\\
         \includegraphics[scale=0.38]{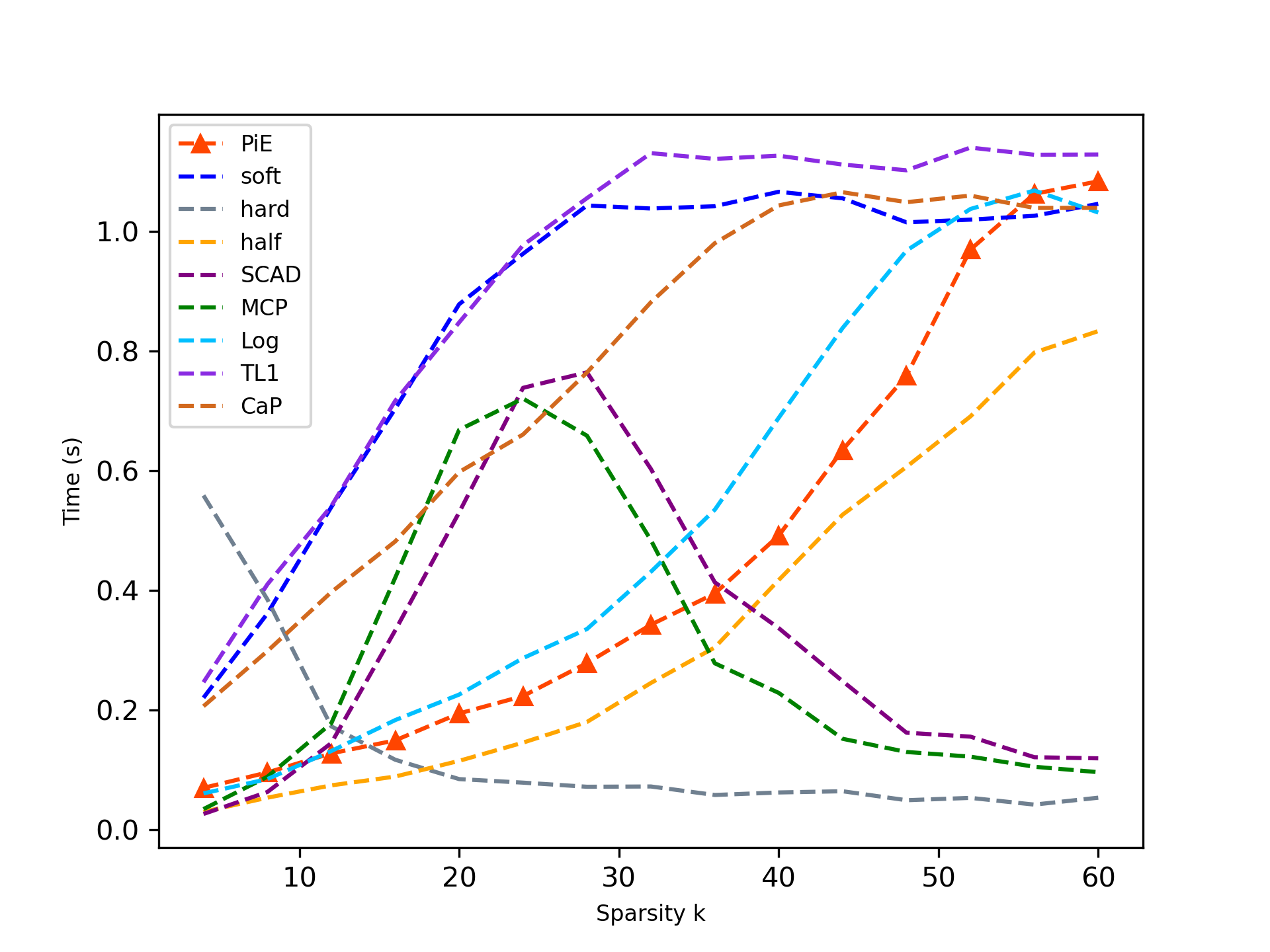}
         \end{minipage}}
\hfil
\subfloat[]{\begin{minipage}[b]{0.3\textwidth}
\includegraphics[scale=0.38]{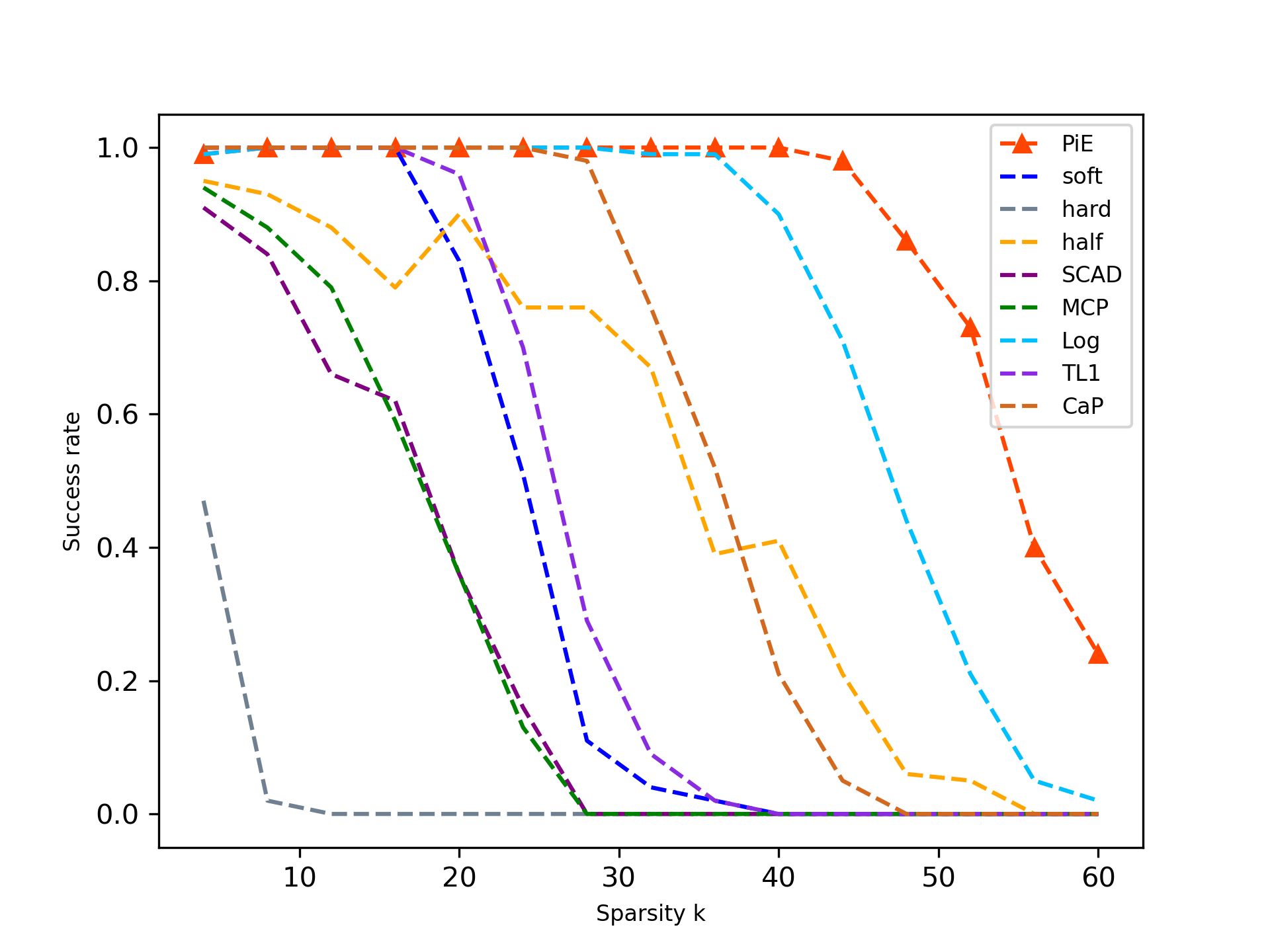}\\
         \includegraphics[scale=0.38]{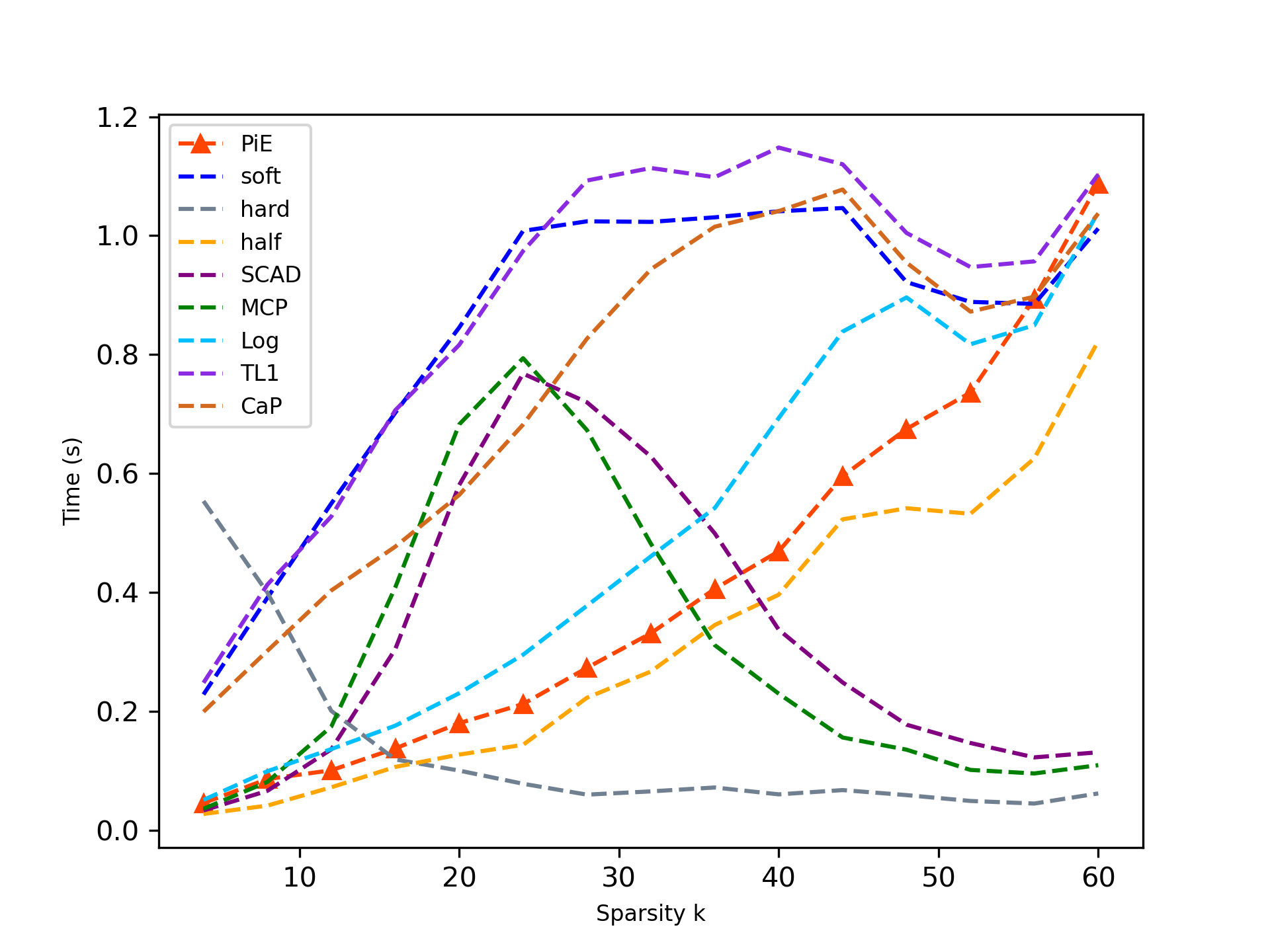}
         \end{minipage}}
\hfil
\subfloat[]{\begin{minipage}[b]{0.3\textwidth}
\includegraphics[scale=0.38]{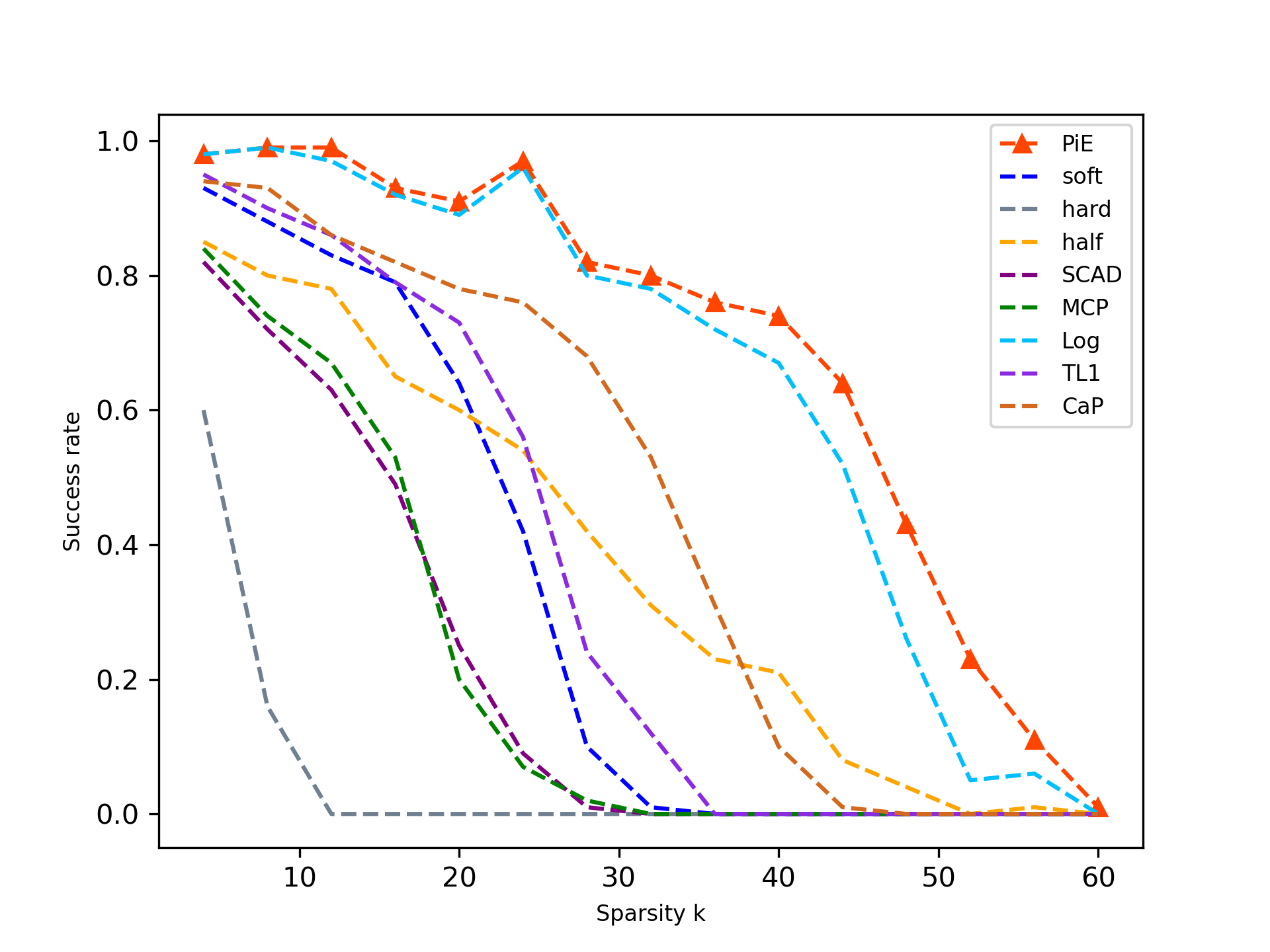}\\
         \includegraphics[scale=0.38]{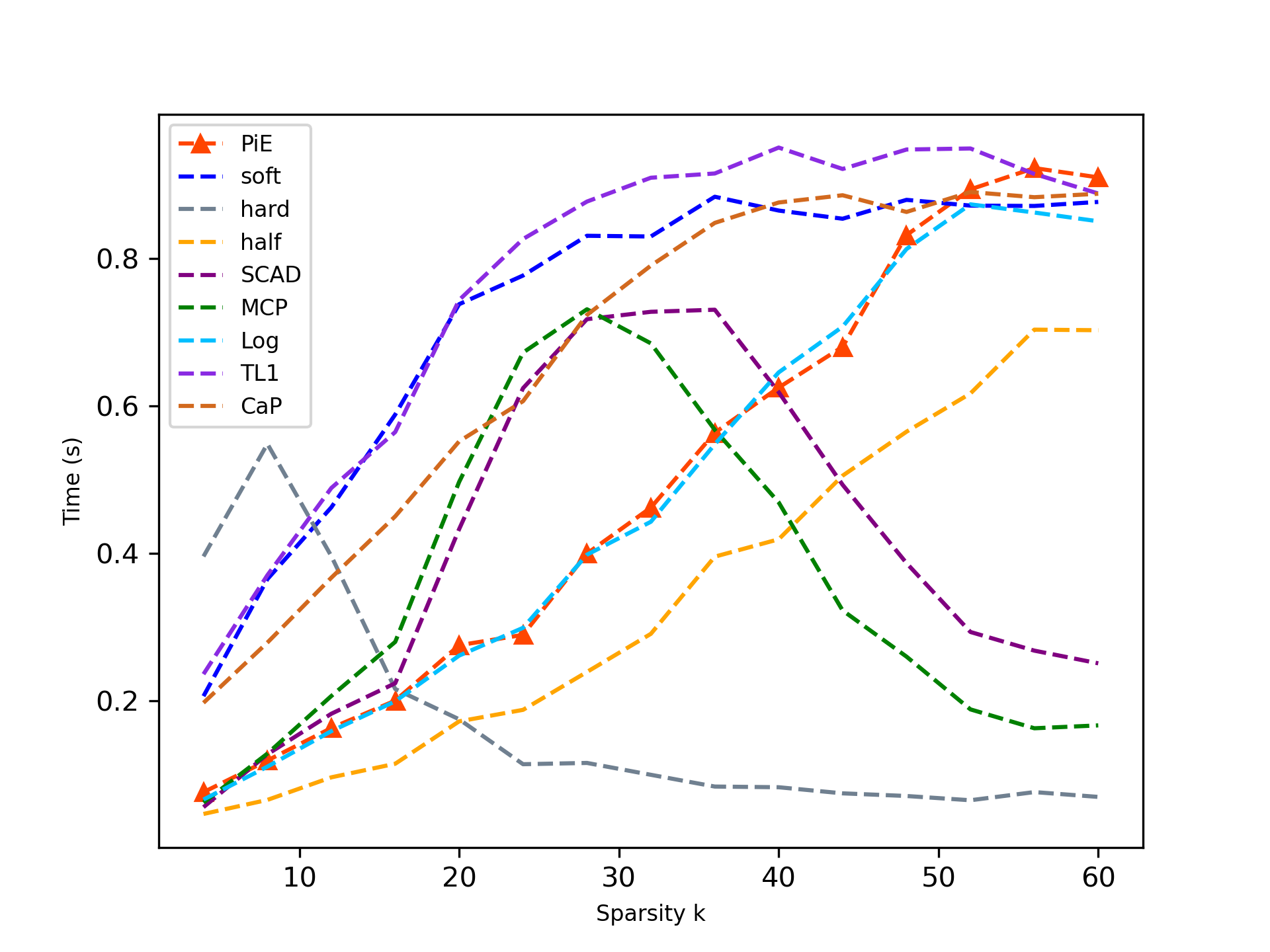}
         \end{minipage}}         
\caption{Success rate (the first row) and convergence time (the second row) of ISTA with a fixed step size $\mu = 0.99\mu_{max}$. (a) Gaussian matrix; (b) DCT matrix with $F=3$; (c) DCT matrix with $F=10$.}
\label{Fig11}
\end{figure*}

To conclude this subsection, we should make some comments.
The comparative study of nonconvex penalties via ISTA in compressed sensing has been given in \cite{Xu2023}. The numerical performance in Fig. \ref{Fig10} and Fig. \ref{Fig11} is significantly better than the existing numerical results in \cite[Figure 9 (a) and (c)]{Xu2023} whenever the related parameters are chosen properly, such as, the number of the maximum iteration and the step size $\mu$.  Notice that an additionally debiased trick in \cite[Figure 9 (c)]{Xu2023} was done by the least squares on the support of the reconstructed signal via ISTA.
The number of the maximum iteration in \cite{Xu2023} is $500$, while it is $3000$ in this paper. 
The step size $\mu=1$ was fixed in \cite{Xu2023}, which can not guarantee the convergence of ISTA as $\mu>\mu_{\max}$. It was reported in \cite{Zhang2018} that the set of matrices satisfying the generalized null space property of concave penalties can be larger than that of the null space property. Thus, the exact recovery by Half, TL1, PiE, MCP is potentially better than Soft, CaP, SCAD in this regime \cite[Remark 2.2]{Zhang2018}.

\section{Conclusion}\label{Section:Conclusion}

In this paper, we have derived a correct and simple expression of the proximal operator for the PiE function. Theorems \ref{ProximalCase1} and \ref{ProxTheoCase2Refined} have showed that the proximal operator for the PiE penalty has closed form, while Theorem \ref{Coratau} as an alternative of Theorem \ref{ProxTheoCase2Refined} is better suitable for use in numerical calculations.
Given a sparse signal such that amplitudes of its nonzero coefficients are uniformly distributed on $[-5,5]$ and the sensing matrix $A\in\bR^{128\times256}$, 
numerical experiments in compressed sensing in Section \ref{Section:CS} suggest
 the  three parameters in Algorithm \ref{AlgoISTA} should be given $\lambda=\frac{1}{100}, \sigma=\frac{1}{2}$, and $\mu=0.99\mu_{\rm max}$. These numerical experiments also
have demonstrated the great advantage of ISTA with PiE compared with other penalties (soft, hard, half, SCAD, MCP, Log, TL1, and CaP) whenever the large step size was adopted.
We hope to apply PiE penalty function to other sparse-promoting optimization problems such as support vector machines, neural networks, low-rank matrix recovery, Fourier transform, and wavelet transform, etc in future work.

\section*{Acknowledgments}
The first author was supported by Guangdong Basic and Applied Basic Research Foundation (2023A1515012891, 2020A1515010408). The third author was supported in part by the National Natural Science Foundation of China (11901595,11971490), by the Guangzhou Science and Technology Plan Project (202201010677), and by the Opening Project of Guangdong Province Key Laboratory of Computational Science at the Sun Yat-sen University (2021018).

\bibliographystyle{IEEEtran}
\bibliography{Proximalexp}




\vfill

\end{document}